\documentclass[12pt, reqno]{amsart}
\usepackage{amsmath, amsthm, amscd, amsfonts, amssymb, graphicx, color}
\usepackage[bookmarksnumbered, colorlinks, plainpages]{hyperref}
\usepackage{accents}
\makeatletter
     \def\section{\@startsection{section}{1}%
      \z@{.7\linespacing\@plus\linespacing}{.5\linespacing}%
     {\bfseries
     \centering
     }}
     \def\@secnumfont{\bfseries}
     \makeatother
\newtheorem{thm}{Theorem}[section]
\newtheorem{lem}[thm]{Lemma}
\newtheorem{prop}[thm]{Proposition}
\newtheorem{cor}[thm]{Corollary}
\theoremstyle{definition}

\newtheorem{example}{Example}

\theoremstyle{remark}
\newtheorem{rem}[thm]{Remark}
\numberwithin{equation}{section}

\def\title#1{{\Large\bf  \begin{center} #1 \vspace{0pt} \end{center}  } }
\def\authors#1{{\large\bf \begin{center} #1 \vspace{0pt} \end{center} } }
\def\university#1{{\sl \begin{center} #1 \vspace{0pt} \end{center} } }
\def\inst#1{\unskip $^{#1}$}
\newcommand{\vertiii}[1]{{\left\vert\kern-0.25ex\left\vert\kern-0.25ex\left\vert 
#1 \right\vert\kern-0.25ex\right\vert\kern-0.25ex\right\vert}}
\usepackage{amsfonts}

\def \d  {\text {\rm d}}
\def \det   {\text {\rm det}}
\def \diag  {\text {\rm diag}}

\begin{document}
%
%
%

\title{Derivatives of  symplectic eigenvalues and a Lidskii type theorem}

\bigskip

%
%

\authors{Tanvi Jain\inst{1}, 
Hemant Kumar Mishra\inst{2} 
}


\smallskip
%
%

\university{\inst{1} Indian Statistical Institute, New Delhi 110016, India\\tanvi@isid.ac.in}
\university{\inst{2} Indian Statistical Institute, New Delhi 110016, 
India\\hemantmishra1124@gmail.com}%
\date{\today}

\begin{abstract}
Associated with every $2n\times 2n$ real positive definite matrix $A,$
there exist $n$ positive numbers called the symplectic eigenvalues of $A,$
and a basis of $\mathbb{R}^{2n}$ called the symplectic eigenbasis of $A$ corresponding to these numbers.
In this paper, we discuss the differentiability (analyticity) of the symplectic eigenvalues and corresponding symplectic eigenbasis for
differentiable (analytic) map $t\mapsto A(t),$ and compute their derivatives.
We then derive an analogue of Lidskii's theorem for symplectic eigenvalues as an application.
\end{abstract}

\makeatletter
\@setabstract
\makeatother

\vskip0.3in
\footnotetext{\noindent {\bf AMS Subject Classifications:} 15A48, 15A18, 15A45, 15A90, 81P45, 81S10.}

 \footnotetext{\noindent {\bf Keywords : }Positive definite matrix, Williamson's theorem, symplectic eigenvalue, symplectic eigenvector pair, derivative, analyticity, majorisation, Lidskii's theorem.}

%
%

\section{Introduction}

Let $J$ be the $2n\times 2n$ matrix
\begin{equation}
J=\begin{bmatrix}O & I_n\\
-I_n & O\end{bmatrix},\label{eqflr01}
\end{equation}
 where $I_n$ is the $n\times n$ identity matrix.
A $2n\times 2n$ real matrix $M$ is called a {\it symplectic matrix} if
$$M^TJM=J.$$
The set of all symplectic matrices forms a group under multiplication and is denoted by $Sp(2n).$
A result on symplectic matrices, generally known as {\it Williamson's theorem} says that
for every $2n\times 2n$ real positive definite matrix $A$ there exists a symplectic matrix $M$ such that
\begin{equation}
M^TAM=\begin{bmatrix}D & O\\
O & D\end{bmatrix},\label{eq3}
\end{equation}
where $D$ is an $n\times n$ positive diagonal matrix with diagonal entries $d_1(A)\le \cdots \le d_n(A),$ \cite{dms, degosson}.
The positive numbers $d_1(A),\dots,  d_n(A)$ are uniquely determined.
We call these numbers the {\it symplectic eigenvalues} of $A.$
These are the complete invariants of $A$ under the action of the symplectic group $Sp(2n).$
Symplectic eigenvalues occur in different areas of mathematics and physics such as
symplectic geometry, symplectic topology and both classical and quantum mechanics. 
See \cite{sanders, degosson, koenig, safranek}.
Recently there has been a heightened interest in the study of symplectic eigenvalues by both physicists and mathematicians.
A particular reason for this being their growing importance and applications
in quantum information.
See, for instance, \cite{dms, p}.

A positive number $d$ is a symplectic eigenvalue of $A$ if and only if $\pm d$
is an eigenvalue of the Hermitian matrix $\imath A^{1/2}JA^{1/2},$ \cite{degosson, p}.
In principle, it could be possible to derive the properties of symplectic eigenvalues from the well-known properties of eigenvalues of Hermitian matrices.
But due to the complicated form of the Hermitian matrix $\imath A^{1/2}JA^{1/2},$
it is often not feasible to obtain results for symplectic eigenvalues
from the well-developed theory for eigenvalues of Hermitian matrices.
So, it is necessary as well as helpful to develop
independent techniques and theory for symplectic eigenvalues.
Some fundamental inequalities and variational principles on symplectic eigenvalues are given in \cite{bj}.
In this paper we study some questions on symplectic eigenvalues
analogous to some fundamental questions on eigenvalues of Hermitian matrices that have been studied for long.

Eigenvalue problems for Hermitian matrices have a long and rich history.
We can classify these problems to be qualitative and quantitative in nature.
An example of qualitative problems is the study of continuity, differentiability and analyticity of eigenvalues and eigenvectors
as functions of Hermitian matrices when the matrices depend smoothly on a parameter.
These problems have been extensively studied,
(see e.g., \cite{juang, kato, kazdan, magnus, rellich, rogers, xu})
and are of much importance in perturbation theory, differential equations, numerical analysis and physics.
See \cite{kilic, seyranian, xu}.
The quantitative problems include variational principles, eigenvalues of functions of matrices,
majorisation inequalities and computation of eigenvalues and eigenvectors.
There has been much interest in the study of relationships between the eigenvalues of Hermitian matrices $A$ and $B$ and those of their sum $A+B.$
Suppose $\lambda^\uparrow(A)=\left(\lambda_1^\uparrow(A),\ldots,\lambda_n^\uparrow(A)\right)$ denote the tuple of eigenvalues
of an $n\times n$ Hermitian matrix $A$ arranged in increasing order.
In 1912 H. Weyl discovered several relationships between the eigenvalues of sums of Hermitian matrices.
These  include the inequalities:
\begin{equation}
\lambda_j^\uparrow(A+B)\ge\lambda_j^\uparrow(A)+\lambda_1^\uparrow(B) \ \ 1\le j\le n.\label{eq1lr}
\end{equation}
The maximum principle given by Ky Fan in 1949 implies that
for all $1\le k\le n,$
\begin{equation}
\sum\limits_{j=1}^{k}\lambda_{j}^\uparrow(A+B)\ge\sum\limits_{j=1}^{k}\lambda_j^\uparrow(A)+\sum\limits_{j=1}^{k}\lambda_j^\uparrow(B).\label{eq2lr}
\end{equation}
In 1950 V. B. Lidskii proved the inequalities
\begin{equation}
\sum\limits_{j=1}^{k}\lambda_{i_j}^\uparrow(A+B)\ge\sum\limits_{j=1}^{k}\lambda_{i_j}^\uparrow(A)+\sum\limits_{j=1}^{k}\lambda_j^\uparrow(B)\label{eq3lr}
\end{equation}
for all $k=1,\ldots,n$ and $1\le i_1<i_2<\cdots <i_k\le n.$
Inequalities \eqref{eq1lr} and \eqref{eq2lr} are special cases of \eqref{eq3lr}.
Lidskii's inequalities played a fundamental role in the study of eigenvalues of sums of matrices and
proved to be an important stimulant for the much celebrated Horn's conjecture. See, for instance, \cite{rbhh,ful}.
These inequalities have attracted much attention and a number of different proofs  for these are now available in literature.
See \cite{rbhp, mathias}.
But all the proofs are generally more difficult than those for the earlier two families of inequalities \eqref{eq1lr} and \eqref{eq2lr}.

In this paper, we address both the qualitative as well as quantitative problems on symplectic eigenvalues.
We study differentiability of symplectic eigenvalues and also derive a relationship analogous to Lidskii's theorem for these numbers.
Let $\mathbb{P}(n)$ denote the set of all $n\times n$ real positive definite matrices.
For a matrix $P$ in $\mathbb{P}(2n),$ we shall always denote by $d_1(P)\le\cdots\le d_n(P),$
its symplectic eigenvalues
{\it arranged in increasing order}.
We know that each map $P\mapsto d_j(P)$ is continuous. See \cite{bj,idel}.
But this map need not be differentiable, as is shown by the following example.

\begin{example}\label{ex1}
Let $I_4$ denote the $4\times 4$ identity matrix. Clearly
$d_1(I_4)=d_2(I_4)=1.$
We show that the maps $P\mapsto d_1(P)$ and $P\mapsto d_2(P)$ are not even Gateaux differentiable at $I_4.$
Let $B$ be the $4\times 4$ matrix
$$B=I_2\otimes\begin{bmatrix}
0 & 0\\
0 & 1\end{bmatrix}.$$
For any real number $t$ with $|t|<1,$
$I_4+tB$ is the matrix
$$I_4+tB=I_2\otimes\begin{bmatrix}
1 & 0\\
0 & 1+t\end{bmatrix}.$$
The symplectic eigenvalues of $I_4+tB,$ are
\begin{equation*}
d_1(I_4+tB)=\begin{cases}
1+t & -1<t<0\\
1 & 0\le t<1,
\end{cases}
\end{equation*}
and
\begin{equation*}
d_2(I_4+tB)=\begin{cases}
1 & -1<t<0\\
1+t & 0\le t<1.
\end{cases}
\end{equation*}
It is easy to see that
$$\lim\limits_{t\to 0+}\frac{d_1(I_4+tB)-d_1(I_4)}{t}=0$$
and
$$\lim\limits_{t\to 0-}\frac{d_1(I_4+tB)-d_1(I_4)}{t}=1.$$
 This shows that the map $d_1$ is not differentiable.
Similarly we can see that the map $d_2$ is not differentiable at $I_4.$
\end{example}

A symplectic eigenvalue $d$ of $A$ has {\it multiplicity $m$} if
the set $\{i:d_i(A)=d\}$ has exactly $m$ elements,
and is {\it simple} if $m=1.$
We see in Example \ref{ex1}, the symplectic eigenvalue $d_1$ of $I_4$ has multiplicity $2$ and is not differentiable at $I_4.$
We show in Theorem \ref{thm_main1} that if $d_j(A)$
is a simple symplectic eigenvalue of $A,$
then the map $P\mapsto d_j(P)$ and the corresponding symplectic eigenvector pair maps are infinitely differentiable at $A.$
We calculate the first derivatives of these maps in Theorem \ref{thm1b}.

We also study the differentiability and analyticity of symplectic eigenvalues
of positive definite matrices that are dependent on a real parameter.
We show in Theorem \ref{thm_main2},
if $t\mapsto A(t)$ is a real analytic map from an open interval to the space of positive definite matrices,
then we can choose all the symplectic eigenvalues and corresponding symplectic eigenbasis to be real analytic in $t.$
We also see that in this case, the maps $t\mapsto d_j\left(A(t)\right)$ are piecewise real analytic.
See Theorem \ref{thm3.2.pa}.

We now describe the quantitative problems that we study in this paper.
Recently there has been much interest in finding relationships between the symplectic eigenvalues of sums of positive definite matrices and
those of individual matrices.
T. Hiroshima in \cite{h} proved the following relationship for symplectic eigenvalues that is analogous to \eqref{eq2lr}.
$$\sum\limits_{j=1}^{k}d_j(A+B)\ge\sum\limits_{j=1}^{k}d_j(A)+\sum\limits_{j=1}^{k}d_j(B)\ \ 1\le k\le n.$$
In \cite{rbhs} R. Bhatia addressed the inequality analogous to \eqref{eq1lr}.
He showed that
$$d_j(A+B)\ge d_j(A)+d_1(B)$$
for all $j=1,\ldots,n$ when $A$ and $B$ are of the form
$$A=\begin{bmatrix}D & O\\
O & D\end{bmatrix},\ B=\begin{bmatrix}X & O\\
O & X^{-1}\end{bmatrix},$$
where $D$ is the diagonal matrix $\diag(d_1(A),\ldots,d_n(A))$ and $X$ is any $n\times n$ positive definite matrix.
As an application of our results on analyticity of symplectic eigenvalues,
we derive relationships analogous to Lidskii's inequalities (Theorem \ref{thm1l}).
More precisely, we show that for all $k=1,\ldots,n$ and all $1\le i_1<\cdots< i_k\le n,$
$$\sum\limits_{j=1}^{k}d_{i_j}(A+B)\ge\sum\limits_{j=1}^{k}d_{i_j}(A)+\sum\limits_{j=1}^{k}d_j(B).$$
As for the case of eigenvalues of Hermitian matrices, these greatly generalise the inequalities given in \cite{rbhs} and \cite{h}.
In this process, we introduce a notion similar to the notion of projections, that we call as ``symplectic projections'', and
 give an equivalent statement for Williamson's theorem in terms of symplectic projections.

The paper is organised as follows.
Some definitions and preliminary results on symplectic eigenvalues are summarised in Section 2.
In Section 3, we study the differentiability of symplectic eigenvalues and symplectic eigenvectors maps when the symplectic eigenvalues are simple,
and compute their first order derivatives.
In Section 4 we discuss differentiability and analyticity of these maps for curves of positive definite matrices
when the symplectic eigenvalues are not
necessarily simple.
As applications of our results, we derive a symplectic analogue of Lidskii's theorem
and give a perturbation bound in Section 5.

\section{Preliminaries}


Let $\mathbb{R}^{n}$ denote the space of all $n$ tuples over the real numbers,
and let $\mathbb{M}(n)$ denote the space of all $n\times n$ real matrices.
The bilinear form $(\cdot,\cdot)$ on $\mathbb{R}^{2n}$ given by
\begin{equation}
(x,y)=\sum\limits_{i=1}^{n}(x_iy_{n+i}-x_{n+i}y_i)\label{eq1}
\end{equation}
is called the {\it symplectic inner product} on $\mathbb{R}^{2n}.$
This can also be written as
$$(x,y)=\langle x,Jy\rangle.$$
Here $J$ is the $2n\times 2n$ matrix given by \eqref{eqflr01},
and $\langle \cdot,\cdot\rangle$ denotes the Euclidean inner product on $\mathbb{R}^{2n}.$
It is easy to see that a matrix $M\in Sp(2n)$ if and only if it preserves the symplectic inner product on $\mathbb{R}^{2n},$ i.e.,
$$(Mx,My)=\langle Mx,JMy\rangle=\langle x,Jy\rangle=(x,y).$$
A pair of vectors $(u,v)$ is called {\it normalised} if $\langle u,Jv\rangle=1.$
Two pairs of vectors $(u_1,v_1)$ and $(u_2,v_2)$ are called {\it symplectically orthogonal} if
\begin{equation}
\langle u_i,Jv_j\rangle=\langle u_i,Ju_j\rangle=\langle v_i,Jv_j\rangle=0\label{eq2}
\end{equation}
for $i\ne j,$ $i,j=1,2.$
A subset $\{u_1,\ldots,u_m,v_1,\ldots,v_m\}$
of $\mathbb{R}^{2n}$ is called a {\it symplectically orthogonal (orthonormal)} set if the pairs of vectors $(u_i,v_i)$ are mutually symplectically orthogonal (and normalised).
If $m=n,$ then the symplectically orthonormal set is called a {\it symplectic basis} of $\mathbb{R}^{2n}.$

The following proposition is an easy consequence of Williamson's Theorem.

\begin{prop}\label{prop1}
Let $A$ be a $2n\times 2n$ real positive definite matrix with symplectic eigenvalues $d_1,\ldots,d_n.$
There exists a symplectic basis $\{u_1,\ldots,u_n,v_1,\ldots,v_n\}$ of $\mathbb{R}^{2n}$
such that for each $i=1,\ldots,n,$
\begin{equation}
Au_i=d_iJv_i,\ Av_i=-d_iJu_i,\label{eq4}
\end{equation}

\end{prop}

A pair of vectors $(u_i,v_i)$ that satisfies \eqref{eq4} is called a {\it symplectic eigenvector pair} of $A$ corresponding to the symplectic eigenvalue $d_i.$
If the pair, in addition, is normalised, it is called a
{\it normalised symplectic eigenvector pair} of $A.$

The proofs of the next two results are straightforward and left to the reader.
\begin{lem}\label{lem4}
Let $A\in\mathbb{P}(2n),$ and let $d$ be a positive number.
The following statements are equivalent.
\begin{itemize}

\item[(i)] $d$ is a symplectic eigenvalue of $A$
and $(u,v)$ is a corresponding symplectic eigenvector pair.
\item[(ii)] $\pm d$ is an eigenvalue of $\imath JA$
and $u\mp \imath v$ is a corresponding eigenvector.
\item[(iii)] $\pm d$ is an eigenvalue
of $\imath A^{1/2}JA^{1/2}$ and $A^{1/2}u\mp \imath A^{1/2}v$ is a corresponding eigenvector.
\end{itemize}
\end{lem}

Since $d_1,\ldots,d_n$ denote the symplectic eigenvalues arranged in increasing order,
we usually denote any collection of symplectic eigenvalues by $\tilde{d}_1,\ldots,\tilde{d}_n.$

\begin{prop}\label{prop1sev}
For $A$ in $\mathbb{P}(2n),$
the set $\{(\tilde{u}_j,\tilde{v}_j):j=1,\ldots,m\}$ is a symplectically orthogonal set of symplectic eigenvector pairs of $A$ corresponding to the symplectic eigenvalues
$\tilde{d}_1,\ldots,\tilde{d}_m,$ respectively, if and only if
$\{A^{1/2}\tilde{u}_j- \imath A^{1/2}\tilde{v}_j:j=1,\ldots,m\}$ is an orthogonal set of eigenvectors of $\imath A^{1/2}JA^{1/2}$ corresponding to the eigenvalues $\tilde{d}_1,\ldots,\tilde{d}_m$ respectively.
Further, for each $j=1,\ldots,k$
\begin{equation}
\|A^{1/2}\tilde{u}_j-\imath A^{1/2}\tilde{v}_j\|^2=2\tilde{d}_j\langle \tilde{u}_j,J\tilde{v}_j\rangle.\label{eq1sev}
\end{equation}
\end{prop}


\begin{cor}\label{cor3sev}
Any two symplectic eigenvector pairs corresponding to two distinct symplectic eigenvalues of a real positive definite matrix are symplectically orthogonal.
\end{cor}

\begin{prop}\label{prop1b}
Let $A$ be a $2n\times 2n$ real positive definite matrix, and
let $d$ be a symplectic eigenvalue of $A$ with multiplicity $m.$
Let $r_0=\min \{|d-\tilde{d}|: \tilde{d} \text{ is a symplectic eigenvalue of } A, \tilde{d}\ne d\}.$
Then for any positive number $r< r_0,$
there exists an open neighbourhood $U$ of $A$ in $\mathbb{P}(2n)$
such that every $P$ in $U$ has exactly $m$ symplectic eigenvalues (counted with multiplicities)
contained in $(d-r,d+r).$
\end{prop}
 
\begin{proof}
Let $d_1(A)\le\cdots\le d_i(A)<d_{i+1}(A)=\cdots=d_{i+m}(A)<d_{i+m+1}(A)\le \cdots\le d_n(A)$
be the $n$ symplectic eigenvalues of $A$ with $d_{i+1}(A)=\cdots=d_{i+m}(A)=d.$
By our choice of $r$ we see that
$$d_i(A)<d-r<d+r<d_{i+m+1}(A).$$
Since each $d_j$ is continuous,
we can find an open neighbourhood $U$ of $A$ such that for every $P\in U,$
$$d_{i+1}(P),\ldots,d_{i+m}(P)\in (d-r,d+r),$$
$$d_i(P)<d-r\textrm{ and }d_{i+m+1}(P)>d+r.$$
Thus for every $P\in U,$ there are exactly $m$ symplectic eigenvalues $d_{i+1}(P),\ldots,d_{i+m}(P)$ of $P$ that are contained in $(d-r,d+r).$
The cases $d=d_1$ and $d=d_n$ can be proved in a similar way.
\end{proof}

A subspace $W$ of $\mathbb{R}^{2n}$ is called a {\it symplectic subspace} of $\mathbb{R}^{2n}$ if for every $x\in W$ there exists a $y\in W$
such that $\langle x,Jy\rangle\ne 0.$
(See \cite{degosson} Section 1.2.1.)
If $W$ is a symplectic subspace of $\mathbb{R}^{2n},$
then its dimension is an even number and
there exists a symplectically orthonormal set that spans it.
Let $d$ be a symplectic eigenvalue of $A,$
and let $S$ be the set of all
symplectic eigenvector pairs of $A$ corresponding to $d.$
Suppose $W$ is the span of the set
$\{u,v:(u,v)\in S\}.$
It is easy to see that $W$ is a symplectic subspace of $\mathbb{R}^{2n}.$
If $d$ has multiplicity $k,$
then the dimension of $W$ is $2k.$
\vskip.1in
We end this section with an observation on the extension of Williamson's theorem and the notion of symplectic eigenvalues to positive semidefinite matrices.

\begin{rem}
Let $A$ be a $2n\times 2n$ real positive semidefinite matrix. Then there exists
a symplectic matrix $M$ such that \eqref{eq3} holds
for some $n\times n$ nonnegative diagonal matrix $D$ if and only if
the kernel of $A$ is a symplectic subspace of $\mathbb{R}^{2n}.$
If $\textrm{dim}\, \textrm{Ker}\, A=2m,$
then exactly $m$ diagonal entries of $A$ are zero.
In this case, we call the nonnegative diagonal entries of $D$ to be the symplectic eigenvalues of the positive semidefinite matrix $A.$
\end{rem}

Let $\mathbb{P}_s(2n)$ be the set of all $2n\times 2n$ real positive semidefinite matrices $A$
such that $\textrm{Ker} \, A$ is a symplectic subspace of $\mathbb{R}^{2n}.$
We can see from the proof of Theorem 7 of \cite{bj}.
that the maps $d_j$ taking $A$ to $d_j(A)$ are continuous on $\mathbb{P}_s(2n)$ for all $j=1,\ldots,n.$

\section{Simple symplectic eigenvalues}

The following theorem is the key result that will be used to prove the main theorem of this section.

\begin{thm}\label{thm2}
Let $A$ be a $2n\times 2n$ real positive definite matrix.
Suppose $d_0$ is a simple symplectic eigenvalue of $A$ with corresponding normalised symplectic eigenvector pair $(u_0,v_0).$
Then there exists an open subset $U$ of $\mathbb{P}(2n)$ containing $A,$
and $C^\infty$ maps $d:U\to\mathbb{R}$ and $u,v:U\to \mathbb{R}^{2n}$ that satisfy the following conditions.
\begin{itemize}
\item[(i)] For every $P\in U,$ $d(P)$ is a simple symplectic eigenvalue of $P$ with the corresponding normalised symplectic eigenvector pair $(u(P),v(P)).$
\item[(ii)] $d(A)=d_0,$ $u(A)=u_0$ and $v(A)=v_0.$
\item[(iii)] \begin{equation}
\langle u_0,Ju(P)\rangle+\langle v_0,Jv(P)\rangle=0.\label{eq1c}
\end{equation}
\end{itemize}
\end{thm}

\begin{proof}
Since $d_0$ is a simple symplectic eigenvalue of $A$ with symplectic eigenvector pair $(u_0,v_0),$
by $\text{Lemma \ref{lem4}},$ it is a simple eigenvalue of $\imath JA$ with eigenvector $x_0=u_0-\imath v_0.$
Also $\langle x_0,Jx_0\rangle=-2\imath \langle u_0,Jv_0\rangle=-2\imath .$
Define the map $\varphi:\mathbb{P}(2n)\times\mathbb{C}^{2n}\times \mathbb{C}\to\mathbb{C}^{2n}\times \mathbb{C}$ as
$$\varphi(P,x,d)=\left((\imath JP-d)x,\langle x_0,Jx\rangle+2\imath \right).$$
Clearly, $\varphi$ is a $C^\infty$ map and $\varphi(A,x_0,d_0)=0.$
Let $D_2\varphi$ denote the partial derivative of $\varphi$ with respect to $(x,d).$ 
Then
$$D_2\varphi(A,x_0,d_0)=\begin{bmatrix} \imath JA-d_0 & -x_0\\
x_0^*J & 0\end{bmatrix}.$$
Thus $\det\, D_2\varphi(A,x_0,d_0)=-\langle x_0,J(\imath JA-d_0)^{\textrm{adj}}x_0\rangle.$
Since $d_0$ is a simple eigenvalue of $\imath JA,$ $0$ is a simple eigenvalue of $\imath JA-d_0.$
So we have $(\imath JA-d_0)^{\textrm{adj}}x_0=cx_0,$ where $c$ is the product of all nonzero eigenvalues of $\imath JA-d_0.$
This gives
$$\langle x_0,J(\imath JA-d_0)^{\textrm{adj}}x_0\rangle=c\langle x_0,Jx_0\rangle=-2\imath c\ne 0.$$
Hence by the Implicit function theorem,
there exists an open subset $U$ of $\mathbb{P}(2n)$ containing $A,$
and $C^\infty$ maps $d:U\to\mathbb{C}$ and $x:U\to\mathbb{C}^{2n}$ that satisfy
$\imath JPx(P)=d(P)x(P),$ $\langle x_0,Jx(P)\rangle=-2\imath,$
$x(A)=x_0$ and $d(A)=d_0.$
Clearly $x(P)\ne 0,$ and hence $d(P)$ is an eigenvalue of $\imath JP.$
All eigenvalues of $\imath JP$ are real.
Hence $d(P)$ is real.
Since $d_0>0,$ we can assume that $d(P)>0$ for all $P\in U.$
By Lemma \ref{lem4}, we see that $d(P)$ is a symplectic eigenvalue of $P$ for every $P\in U.$
Also since $D_2\varphi(P,x(P),d(P))$ is invertible, $(\imath JP-d(P))^{\textrm{adj}} \neq 0$ and this implies that $d(P)$ has multiplicity $1.$
Let $x(P)=\tilde{u}(P)-\imath \tilde{v}(P)$ be the Cartesian decomposition of $x(P).$
By $\text{Lemma } \ref{lem4}$
we see that $(\tilde{u}(P),\tilde{v}(P))$ is a symplectic eigenvector pair of $P$ corresponding to $d(P).$
Also, the maps $P\mapsto \tilde{u}(P)$ and $P\mapsto \tilde{v}(P)$ are $C^\infty$ on $U,$
and $\tilde{u}(A)=u_0$ and $\tilde{v}(A)=v_0.$
We know that $\langle u_0,Jv_0\rangle=1.$
Hence we can assume that $\langle \tilde{u}(P),J\tilde{v}(P)\rangle>0$ for all $P\in U.$
This implies that the map $P\mapsto \langle \tilde{u}(P),J\tilde{v}(P)\rangle^{-1/2}$ is $C^\infty$ on $U.$
Define the maps $u,v:U\to\mathbb{R}^{2n}$ as
$$u(P)=\langle \tilde{u}(P),J\tilde{v}(P)\rangle^{-1/2}\tilde{u}(P)$$  
and
$$ v(P)=\langle \tilde{u}(P),J\tilde{v}(P)\rangle^{-1/2}\tilde{v}(P).$$
The maps $u$ and $v$ are $C^\infty$ and $(u(P),v(P))$ forms a normalised symplectic eigenvector pair of $P$ corresponding to $d(P).$
This shows the existence of infinitely differentiable maps $d,u,v$ on $U$ that satisfy (i) and (ii).
Moreover, since the real part of $\langle x_0,Jx(P)\rangle$ is zero,
$$\langle u_0,Ju(P)\rangle+\langle v_0,Jv(P)\rangle=0$$
This proves (iii).
\end{proof}

\begin{rem}
Since $d_0$ is a simple symplectic eigenvalue of $A$ if and only if it is a simple eigenvalue of $\imath A^{1/2}JA^{1/2},$ (see Proposition \ref{prop1sev})
and the square root map is infinitely differentiable on real positive definite matrices,
we can obtain (i) and (ii) of Theorem \ref{thm2} from the corresponding result on eigenvalues in \cite{rellich}.
But we give an independent proof as \eqref{eq1c} is required in the
computation of the derivatives of symplectic eigenvector pair in Theorem \ref{thm1b}.
\end{rem}

The main theorem of this section is as follows:

\begin{thm}\label{thm_main1}
Let $A\in\mathbb{P}(2n),$ and
suppose that $d_j(A)$ is simple.
Then there exists a neighbourhood $U$ of $A$ in $\mathbb{P}(2n)$ such that
for every $P\in U,$ $d_j(P)$ is simple and the map $P\mapsto d_j(P)$ is infinitely differentiable on $U.$
Further, if $(u_0,v_0)$ is a normalised symplectic eigenvector pair of $A$ corresponding to $d_j(A),$
then there exist infinitely differentiable maps $u_j,v_j:U\to\mathbb{R}^{2n}$ such that for every $P$ in $U$ $(u_j(P),v_j(P))$
is a normalised symplectic eigenvector pair of $P$ corresponding to $d_j(P),$
$u_j(A)=u_0$ and $v_j(A)=v_0,$
and $u_j(P),v_j(P)$ satisfy \eqref{eq1c}. 
\end{thm}

\begin{proof}
If $d_j(A)$ is a simple symplectic eigenvalue of $A,$
then by Theorem \ref{thm2},
we can find an open neighbourhood $V$ of $A$ in $\mathbb{P}(2n),$
and $C^\infty$ maps $d:V\to\mathbb{R}$ and $u,v:V\to\mathbb{R}^{2n}$ that satisfy (i)-(iii) of Theorem \ref{thm2};
i.e., 
$d(P)$ is a simple symplectic eigenvalue of $P$
and $(u(P),v(P))$ is a corresponding normalised symplectic eigenvector pair such that
 $d(A)=d_j(A),$ $u(A)=u_0, v(A)=v_0,$
and $u(P),v(P)$ satisfy \eqref{eq1c}.
Let $r$ be a positive number with $r<\min\{d_{j+1}(A)-d_j(A),d_j(A)-d_{j-1}(A)\}.$
By the continuity of the map $P\mapsto d(P)$ and Proposition \ref{prop1b}, we can assume that for every $P$ in $V,$
$d(P)$ is the only symplectic eigenvalue of $P$ contained in $(d_j(A)-r,d_j(A)+r).$
By (\cite{bj}, Theorem 7), we know that the map $P\mapsto d_j(P)$ is continuous.
Hence there exists an open neighbourhood $W$ of $A$ such that $d_j(P)\in (d_j(A)-r,d_j(A)+r)$
for every $P$ in $W.$
But this implies that $d(P)=d_j(P)$ for every $P\in V\cap W.$
Take $U=V\cap W.$
Hence the map $d_j$ is infinitely differentiable on $U$ with the corresponding
normalised symplectic eigenvector maps $u,v$ that satisfy the required conditions. 
\end{proof}

Next we compute the derivatives of the symplectic eigenvalue map $d_j$
and its corresponding symplectic eigenvector pair map at $A$
when $d_j(A)$ has multiplicity $1.$
We note here that if $(u,v)$ is a normalised symplectic eigenvector pair of $A$ corresponding to a simple symplectic eigenvalue $d,$
then any normalised symplectic eigenvector pair $(x,y)$ corresponding to $d$ is of the form
$$x=au-bv\textrm{ and }y=bu+av$$
where $a,b$ are real numbers satisfying $a^2+b^2=1.$

\begin{thm}\label{thm1b}
Let $A\in\mathbb{P}(2n)$ be such that $d_j(A)$ is simple,
and let $(u_j,v_j)$ be a normalised symplectic eigenvector pair map through $(u_j(A),v_j(A))$
obtained from Theorem \ref{thm_main1}.
Suppose $M$ is any symplectic matrix given by \eqref{eq3}.
Then the derivatives $Dd_j(A),$ $Du_j(A)$ and $Dv_j(A)$ at a $2n \times 2n$ symmetric matrix $B$ are given by
\begin{equation}
Dd_j(A)(B)=\frac{\langle u_j(A),Bu_j(A)\rangle+\langle v_j(A),Bv_j(A)\rangle}{2},\label{eq6}
\end{equation}
\begin{equation}
Du_j(A)(B)=M\hat{D}M^TBu_j(A)+M\overline{D}JM^TBv_j(A),\label{eq7}
\end{equation}
and
\begin{equation}
Dv_j(A)(B)=M\hat{D}M^TBv_j(A)-M\overline{D}JM^TBu_j(A),\label{eq8}
\end{equation}
where $\hat{D}$ and $\overline{D}$ are the $2n\times 2n$ diagonal matrices with respective diagonal entries given by
\begin{equation}
\left(\hat{D}\right)_{kk}=\begin{cases}
\frac{d_k(A)}{d_j^2(A)-d_k^2(A)} & k\ne j,\, 1\le k\le n\\
-\frac{1}{4d_j(A)} & k=j,\, 1\le k\le n\\
\left(\hat{D}\right)_{ii} & k=n+i,\, 1\le i\le n,
\end{cases}\label{eq9}
\end{equation}
and
\begin{equation}
\left(\overline{D}\right)_{kk}=\begin{cases}
\frac{d_j(A)}{d_j^2(A)-d_k^2(A)} & k\ne j,\, 1\le k\le n\\
\frac{1}{4d_j(A)} & k=j,\, 1\le k\le n\\
\left(\overline{D}\right)_{ii}  & k=n+i,\, 1\le i\le n.
\end{cases}\label{eq10}
\end{equation}
\end{thm}

\begin{proof}
Since $d_j(A)$ is simple, by Theorem \ref{thm_main1}, we know that the map $d_j$ is infinitely differentiable at $A.$
Since $(u_j,v_j)$ is a normalised symplectic eigenvector pair map obtained from Theorem \ref{thm_main1},
we have
\begin{align}
Pu_j(P)&=d_j(P)Jv_j(P),\label{eq11}\\
Pv_j(P)&=-d_j(P)Ju_j(P),\label{eq12} \\
\langle u_j(P),Jv_j(P)\rangle&=1,\label{eq13}\\
\langle u_j(A),Ju_j(P)\rangle+\langle v_j(A),Jv_j(P)\rangle&=0.\label{eq16}
\end{align}
Differentiating \eqref{eq11} and \eqref{eq12} at $A,$
we see that for every $2n \times 2n$ real symmetric matrix $B$
\begin{equation}
Bu_j(A)+ADu_j(A)(B)=Dd_j(A)(B)Jv_j(A)+d_j(A)JDv_j(A)(B),\label{eq14}
\end{equation}
and
\begin{equation}
Bv_j(A)+ADv_j(A)(B)=-Dd_j(A)(B)Ju_j(A)-d_j(A)JDu_j(A)(B).\label{eq15}
\end{equation}
Taking the inner product of \eqref{eq14} with $u_j(A)$ and using the fact that $\langle u_j(A),Jv_j(A)\rangle=1,$
we get
\begin{align}
 \langle u_j(A),Bu_j(A)\rangle & +\langle u_j(A),ADu_j(A)(B)\rangle \nonumber \\
& =Dd_j(A)(B)+\langle u_j(A),d_j(A)JDv_j(A)(B)\rangle \label{eq13a}
\end{align}
Since
\begin{eqnarray*}
\langle u_j(A),ADu_j(A)(B)\rangle& =&\langle Au_j(A),Du_j(A)(B)\rangle\\
&=& d_j(A)\langle Du_j(A)(B),Jv_j(A)\rangle,
\end{eqnarray*}
we can write \eqref{eq13a} as
\begin{eqnarray}
Dd_j(A)(B) &=& \langle u_j(A),Bu_j(A)\rangle+d_j(A)\langle Du_j(A)(B),Jv_j(A)\rangle\nonumber\\
& &\ \ -d_j(A)\langle u_j(A),JDv_j(A)(B)\rangle.\label{eq17}
\end{eqnarray}
Similarly, taking the inner product of \eqref{eq15} with $v_j(A),$ we get
\begin{eqnarray}
Dd_j(A)(B)&=&\langle v_j(A),Bv_j(A)\rangle-d_j(A)\langle Du_j(A)(B),Jv_j(A)\rangle\nonumber\\
& &\ \ + d_j(A)\langle u_j(A),JDv_j(A)(B)\rangle.\label{eq18}
\end{eqnarray}
Adding \eqref{eq17} and \eqref{eq18} finally gives \eqref{eq6}.

We next compute the derivatives $Du_j(A)$ and $Dv_j(A).$

Let the columns of $M$ be $\tilde{u}_1,\ldots,\tilde{u_n},\tilde{v_1},\ldots,\tilde{v}_n.$
Clearly these vectors form a symplectic eigenbasis of $\mathbb{R}^{2n}$ corresponding to $A.$
We can express $Du_j(A)(B)$ and $Dv_j(A)(B)$ uniquely as
\begin{equation*}
Du_j(A)(B)=\sum\limits_{k=1}^{n}\alpha_k\tilde{u}_k+\sum\limits_{k=1}^{n}\beta_k\tilde{v}_k
\end{equation*}
and
\begin{equation*}
Dv_j(A)(B)=\sum\limits_{k=1}^{n}\gamma_k\tilde{u}_j+\sum\limits_{k=1}^{n}\delta_k\tilde{v}_k,
\end{equation*}
where $\alpha_k=\langle Du_j(A)(B),J\tilde{v}_k\rangle,$ $\beta_k=-\langle Du_j(A)(B),J\tilde{u}_k\rangle,$\\
$\gamma_k=\langle Dv_j(A)(B),J\tilde{v}_k\rangle$ and $\delta_k=-\langle Dv_j(A)(B),J\tilde{u}_k\rangle$
for all $k=1,\ldots,n.$
Since $d_j(A)$ is simple, we can assume that $\tilde{u}_j=au_j(A)-bv_j(A)$ and $\tilde{v}_j=bu_j(A)+av_j(A)$ for some $a,b \in \mathbb{R}$ with $a^2+b^2=1.$
Thus 
\begin{equation}
\langle \tilde{u}_k,Jv_j(A)\rangle=\langle u_j(A),J\tilde{v}_k\rangle=\delta_{kj} a \label{eq20}
\end{equation}
 and
\begin{equation}
\langle \tilde{u}_k,Ju_j(A)\rangle=\langle \tilde{v}_k,Jv_j(A)\rangle=\delta_{kj} b \label{eq21}
\end{equation}
for all $k=1,\ldots,n.$
Here $\delta_{jk}=0$ if $j \neq k$ and $\delta_{jk}=1$ otherwise.
Taking inner product of \eqref{eq14} with $\tilde{u}_k$ we get
\begin{align*}
& \langle \tilde{u}_k,Bu_j(A)\rangle+\langle \tilde{u}_k,ADu_j(A)(B)\rangle\\
&\ \ =Dd_j(A)(B)\langle \tilde{u}_k,Jv_j(A)\rangle+d_j(A)\langle \tilde{u}_k,JDv_j(A)(B)\rangle.
\end{align*}
Using \eqref{eq20} and the values of $\alpha_k$ and $\delta_k,$ this reduces to 
\begin{equation}
d_k(A)\alpha_k-d_j(A)\delta_k=a Dd_j(A)(B) \delta_{kj}-\langle \tilde{u}_k,Bu_j(A)\rangle.\label{eq22}
\end{equation}
Similarly, taking inner products of \eqref{eq14} with $\tilde{v}_k,$ and of \eqref{eq15} with $\tilde{u}_k$ and $\tilde{v}_k,$ and using \eqref{eq20} and \eqref{eq21}, we obtain the expressions
\begin{align}
d_k(A)\beta_k+d_j(A)\gamma_k&=b Dd_j(A)(B) \delta_{kj}-\langle \tilde{v}_k,Bu_j(A)\rangle,\label{eq23}\\
d_j(A)\beta_k+d_k(A)\gamma_k&=-b  Dd_j(A)(B) \delta_{kj}-\langle \tilde{u}_k,Bv_j(A)\rangle,\label{eq24}\\
-d_j(A)\alpha_k+d_k(A)\delta_k&=aDd_j(A)(B)\delta_{kj}-\langle \tilde{v}_k,Bv_j(A)\rangle.\label{eq25}
\end{align}
Thus for each $k=1,\ldots,n$ we have a system of four linear equations in four unknowns $\alpha_k,\beta_k,\gamma_k$ and $\delta_k.$
When $k\ne j,$  this system is
\begin{align*}
	\begin{bmatrix}
		d_{k}(A)&0&0&-d_{j}(A) \\
		0&d_{k}(A)&d_{j}(A)&0 \\
		0&d_{j}(A)&d_{k}(A)&0 \\
		-d_{j}(A)&0&0&d_{k}(A)
	\end{bmatrix}
	\begin{bmatrix}
		\alpha_{k} \\ \beta_{k} \\ \gamma_{k} \\ \delta_{k}
	\end{bmatrix}
	= -
	\begin{bmatrix}
		\langle \tilde{u}_k,Bu_{j}(A) \rangle \\ \langle \tilde{v}_k,Bu_{j}(A) \rangle \\ \langle \tilde{u}_k,Bv_{j}(A) \rangle \\ \langle \tilde{v}_k,Bv_{j}(A) \rangle
	\end{bmatrix}
\end{align*}
Here $d_{j}(A) \neq d_{k}(A)$ therefore the coefficient matrix above is invertible and left multiplying by the inverse we get
\begin{align*}
	\begin{bmatrix}
		\alpha_{k} \\ \beta_{k} \\ \gamma_{k} \\ \delta_{k}
	\end{bmatrix}
	=
	(d_{j}^2(A)-d_{k}^2(A))^{-1}
	\begin{bmatrix}
		d_{k}(A)&0&0&d_{j}(A) \\
		0&d_{k}(A)&-d_{j}(A)&0 \\
		0&-d_{j}(A)&d_{k}(A)&0 \\
		d_{j}(A)&0&0&d_{k}(A)
	\end{bmatrix}	
	\begin{bmatrix}
		\langle \tilde{u}_k,Bu_{j}(A) \rangle \\ \langle \tilde{v}_k,Bu_{j}(A) \rangle \\ \langle \tilde{u}_k,Bv_{j}(A) \rangle \\ \langle \tilde{v}_k,Bv_{j}(A) \rangle
	\end{bmatrix}
\end{align*}
The solution is thus given by the following equations
\begin{equation}
	\alpha_{k}=\frac{1}{d_{j}^2(A)-d_{k}^2(A)}\left( d_{k}(A)\langle \tilde{u}_k,Bu_{j}(A) \rangle +d_{j}(A) \langle \tilde{v}_k,Bv_{j}(A) \rangle \right),\label{eq26}
\end{equation}
\begin{equation}
	\beta_{k}=\frac{1}{d_{j}^2(A)-d_{k}^2(A)}\left( d_{k}(A)\langle \tilde{v}_k,Bu_{j}(A) \rangle -d_{j}(A) \langle \tilde{u}_k,Bv_{j}(A) \rangle \right),\label{eq27}
\end{equation}
\begin{equation}
	\gamma_{k}=\frac{1}{d_{j}^2(A)-d_{k}^2(A)}\left( d_{k}(A)\langle \tilde{u}_k,Bv_{j}(A) \rangle -d_{j}(A) \langle \tilde{v}_k,Bu_{j}(A) \rangle \right),\label{eq28}
\end{equation}
\begin{equation}
	\delta_{k}=\frac{1}{d_{j}^2(A)-d_{k}^2(A)}\left( d_{k}(A)\langle \tilde{v}_k,Bv_{j}(A) \rangle +d_{j}(A) \langle \tilde{u}_k,Bu_{j}(A) \rangle \right).\label{eq29}
\end{equation}
Now, for $k=j$ we have the following system
\begin{align*}
	\begin{bmatrix}
		d_{j}(A)&0&0&-d_{j}(A) \\
		0&d_{j}(A)&d_{j}(A)&0 \\
		0&d_{j}(A)&d_{j}(A)&0 \\
		-d_{j}(A)&0&0&d_{j}(A)
	\end{bmatrix}
	\begin{bmatrix}
		\alpha_{j} \\ \beta_{j} \\ \gamma_{j} \\ \delta_{j}
	\end{bmatrix}
	= -
	\begin{bmatrix}
		\langle \tilde{u}_j,Bu_{j}(A) \rangle -aDd_j(A)(B)\\ \langle \tilde{v}_j,Bu_{j}(A) \rangle -b Dd_j(A)(B) \\ \langle \tilde{u}_j,Bv_{j}(A) +b Dd_j(A)(B) \rangle \\ \langle \tilde{v}_j,Bv_{j}(A) \rangle-aDd_j(A)(B)
	\end{bmatrix}
\end{align*}
Using the expression for $Dd_j(A)(B),$ the fact that $B$ is symmetric and the relationship between $(\tilde{u}_j, \tilde{v}_j)$ and $(u_j(A), v_j(A))$ one can see that the solution to the above system exists and is given by 
\begin{align}
	\alpha_{j}-\delta_{j} &=\frac{1}{2d_{j}(A)}\left(\langle \tilde{v}_{j}(A),Bv_{j}(A) \rangle -\langle \tilde{u}_{j}(A),Bu_{j}(A) \rangle\right) \label{1rel13}\\
	\beta_{j}+\gamma_{j} &=\frac{-1}{2d_{j}(A)}(\langle \tilde{v}_{j}(A),Bu_{j}(A) \rangle)+\frac{-1}{2d_{j}(A)}(\langle \tilde{u}_{j}(A),Bv_{j}(A) \rangle)\label{1rel14}
\end{align}
Differentiating \eqref{eq13} and \eqref{eq16}, respectively, gives
$$\langle Du_j(A)(B),Jv_j(A)\rangle+\langle u_j(A),JDv_j(A)(B)\rangle=0$$
and $$\langle u_j(A),JDu_j(A)(B)\rangle+\langle v_j(A),JDv_j(A)(B)\rangle=0.$$
These in turn imply $\alpha_j+\delta_j=0$ and $\beta_j-\gamma_j=0.$
Thus
\begin{equation}
\alpha_{j}=-\delta_{j}=\frac{1}{4d_{j}(A)}\left(\langle \tilde{v}_{j}(A),Bv_{j}(A) \rangle -\langle \tilde{u}_{j}(A),Bu_{j}(A) \rangle\right) \label{eq30}
\end{equation}
 and
\begin{equation}
\beta_{j}=\gamma_{j}=\frac{-1}{4d_{j}(A)}(\langle \tilde{v}_{j}(A),Bu_{j}(A) \rangle)+\frac{-1}{4d_{j}(A)}(\langle \tilde{u}_{j}(A),Bv_{j}(A) \rangle).\label{eq31}
\end{equation}
Simplifying the above expressions we get for $k \neq j,$
\begin{eqnarray*}
\alpha_k &=& \frac{1}{d_j^2(A)-d_k^2(A)}\left(d_k^2(A)\langle J\tilde{v}_k,A^{-1}Bu_j(A)\rangle+d_j(A)\langle J\tilde{v}_k,JBv_j(A)\rangle\right).
\end{eqnarray*}
\begin{eqnarray*}
\beta_k &= & -\frac{1}{d_j^2(A)-d_k^2(A)}\left(d_k^2(A)\langle J\tilde{u}_k,A^{-1}Bu_j(A)\rangle+d_j(A)\langle J\tilde{u}_k,JBv_j(A)\rangle\right)
\end{eqnarray*}

\begin{eqnarray*}
\alpha_j &=&-\frac{1}{4}\langle J\tilde{v}_j,A^{-1}Bu_j(A)\rangle+\frac{1}{4d_j(A)}\langle J\tilde{v}_j,JBv_j(A)\rangle
\end{eqnarray*}

\begin{eqnarray*}
\beta_j &= &\frac{1}{4}\langle J\tilde{u}_j,A^{-1}Bu_j(A)\rangle-\frac{1}{4d_j(A)}\langle J\tilde{u}_j,JBv_j(A)\rangle
\end{eqnarray*}

Let $x$ be the $2n$ real vector with components $\alpha_1,\ldots,\alpha_n,\beta_1,\ldots,\beta_n.$
Then we see that $x$ can be written as
$$x=\hat{D}\tilde{D}M^{-1}A^{-1}Bu_j(A)+\overline{D}M^{-1}JBv_j(A),$$
where $\tilde{D}$ is the $2n\times 2n$ diagonal matrix with diagonal entries the symplectic eigenvalues of $A,$
$d_1(A),\ldots,d_n(A),$
$d_1(A),\ldots,d_n(A),$
and $\hat{D}$ and $\overline{D}$ are the diagonal matrices given by \eqref{eq9} and \eqref{eq10}, respectively.
Therefore
\begin{eqnarray*}
Du_j(A)(B)&=&M\hat{D}\tilde{D}M^{-1}A^{-1}Bu_j(A)+M\overline{D}M^{-1}JBv_j(A)\\
&=& M\hat{D}M^TBu_j(A)+M\overline{D}JM^TBv_j(A).
\end{eqnarray*}
The last equality follows from the fact that $M^TAM=\tilde{D}$ and $M^TJM=J.$
This proves \eqref{eq7}.
Similar computations give \eqref{eq8}.
\end{proof}

\begin{rem}\label{remflr}
Let $A\in\mathbb{P}(2n),$
and let $d,u,v$ be maps on a neighbourhood $U$ of $A$ such that $d(P)$ is a symplectic eigenvalue of $P$ and $(u(P),v(P))$ is a pair of normalised symplectic eigenvector.
If $d,u,v$ are differentiable at $A,$
then by following the same steps as those used to prove \eqref{eq6},
we can compute the derivative of $d$ at $A$ as
\begin{equation}
Dd(A)(B)=\frac{1}{2}\left(\langle u(A),Bu(A)\rangle+\langle v(A),Bv(A)\rangle\right).\label{eqflr6}
\end{equation}
\end{rem}

Given a map $t\mapsto A(t)$ from an open interval $I$ to $\mathbb{P}(2n),$
we denote the symplectic eigenvalue $d_j(A(t))$ by $d_j(t),$ $1\le j\le n.$

\begin{cor}\label{cor7}
Let $t\mapsto A(t)$ be a map from  an open interval $I$ to $\mathbb{P}(2n)$ that is infinitely differentiable at $t_0\in I.$
Suppose that $d_j(t_0)$ is simple.
Then there exists an open interval $I_0$ containing $t_0$ such that the map $d_j$ is infinitely differentiable on $I_0.$
If $(u_0,v_0)$ is a corresponding normalised symplectic eigenvector pair of $A(t_0),$
then we can find an infinitely differentiable normalised symplectic eigenvector pair map $(u_j,v_j)$ on $I_0$
corresponding to $d_j(t)$ such that
$(u_j(t_0),v_j(t_0))=(u_0,v_0),$
and $((u_j(t),v_j(t))$ satisfies
$$\langle u_0,Ju_j(t)\rangle+\langle v_0,Jv_j(t)\rangle=0$$
for all $t\in I_0.$
Further, for any symplectic matrix $M$ given by the Williamson theorem for $A(t_0),$
\begin{equation}
d_j^\prime(t)=\frac{\langle u_j(t),A^\prime(t)u_j(t)\rangle+\langle u_j(t),A^\prime(t)u_j(t)\rangle}{2}\ \textrm{ for all }t\in J,\label{eqcor71}
\end{equation}
\begin{equation}
u_j^\prime(t_0)=M\hat{D}M^TA^\prime(t_0)u_0+M\overline{D}JM^TA^\prime(t_0)v_0,\label{eqcor72}
\end{equation}
and
\begin{equation}
v_j^\prime(t_0)=M\hat{D}M^TA^\prime(0)v_0-M\overline{D}JM^TA^\prime(t_0)u_0,\label{eqcor73}
\end{equation}
where $\hat{D}$ and $\overline{D}$ are the diagonal matrices associated with $A(t_0)$ given by \eqref{eq9} and \eqref{eq10}, respectively.
\end{cor}

\begin{thm}
Following the notations of Corollary \ref{cor7}, the second derivative
of $d_j$ at $t_0$ is given by
\begin{align}
d_j^{\prime\prime}(t_0) =& \frac{1}{2}\left(\langle u_0,A^{\prime\prime}(t_0)u_0\rangle+\langle v_0,A^{\prime\prime}(t_0)v_0 \rangle \right) \nonumber \\
&+ 2  \langle A^{\prime}(t_0)u_0, M \overline{D}JM^TA^{\prime}(t_0)v_0\rangle \nonumber \\
& +\langle A^\prime(t_0)u_0,M\hat{D}M^TA^\prime(t_0)u_0\rangle+\langle A^\prime(t_0)v_0, M\hat{D}M^TA^\prime(t_0)v_0\rangle,\label{eqthm81}
\end{align}
where $\hat{D}$ and $\overline{D}$ are the diagonal matrices associated with $A(t_0)$ given by \eqref{eq9} and \eqref{eq10}, respectively.
\end{thm}

\begin{proof}
By \eqref{eqcor71}, we have
\begin{equation}
d_j^\prime(t)=\frac{\langle u_j(t),A^\prime(t)u_j(t)\rangle+\langle v_j(t),A^\prime(t)v_j(t)\rangle}{2}\label{eqpf81}
\end{equation}
for every $t$ in $ I_0.$
Differentiating \eqref{eqpf81} at $t=t_0$ and using the fact that $A^\prime(t_0)$ is real symmetric,  we get
\begin{eqnarray}
d_j^{\prime\prime}(t_0)&=& \frac{1}{2}\left(\langle u_0,A^{\prime\prime}(t_0)u_0\rangle+\langle v_0,A^{\prime\prime}(t_0)u_0\rangle\right)\nonumber\\
& &\ +\langle u_j^\prime(t_0),A^\prime(t_0)u_0\rangle+\langle v_j^\prime(t_0),A^\prime(t_0)v_0\rangle.\label{eqpf81m}
\end{eqnarray}
Using the expression \eqref{eqcor72} for the derivative $u_j^\prime(t_0),$ we get
\begin{align}
\langle u_j^\prime(t_0),A^\prime(t_0)u_0\rangle =& \langle M\hat{D}M^TA^\prime(t_0)u_0,A^\prime(t_0)u_0\rangle \nonumber \\
&+\langle M\overline{D}JM^TA^\prime(t_0)v_0,A^\prime(t_0)u_0\rangle.\label{eqpf82}
\end{align}
Similarly using \eqref{eqcor73}, we have
\begin{align*}
\langle v_j^\prime(t_0),A^\prime(t_0)v_0\rangle =& \langle M\hat{D}M^TA^\prime(t_0)v_0,A^\prime(t_0)v_0\rangle \nonumber \\
&-\langle M\overline{D}JM^TA^\prime(t_0)u_0,A^\prime(t_0)v_0\rangle
\end{align*}
Since $\overline{D}J=J\overline{D},$ we have
\begin{align}
\langle v_j^\prime(t_0),A^\prime(t_0)v_0\rangle= & \langle M\hat{D}M^TA^\prime(t_0)v_0,A^\prime(t_0)v_0\rangle \nonumber \\
&+ \langle M\overline{D}JM^TA^\prime(t_0)v_0,A^\prime(t_0)u_0\rangle.\label{eqpf83}
\end{align}
Using \eqref{eqpf82} and \eqref{eqpf83} in \eqref{eqpf81m}, we obtain \eqref{eqthm81}.
\end{proof}

\section{Symplectic eigenvalues of curves of positive definite matrices}
\sloppy
In this section we study the differentiability and analyticity of symplectic eigenvalues of positive definite matrices dependent on a real parameter
irrespective of their multiplicities.
The matrix square root is an infinitely differentiable map,
and the symplectic eigenvalues of $A$ are the positive eigenvalues of the Hermitian matrix $\imath A^{1/2}JA^{1/2}.$
So, we obtain the results on the differentiability of symplectic eigenvalues
by using the corresponding results on eigenvalues of Hermitian matrices.
We similarly derive the results on analyticity of symplectic eigenvalues.
For details on the differentiability and analyticity of eigenvalues and eigenvectors of curves of Hermitian matrices, see \cite{kato, rellich}.

Let $t\mapsto H(t)$ be a map from an open interval $I$ to the space $\mathbb{H}(k)$ of all $k \times k$ Hermitian matrices
that is differentiable at $t_0\in I.$
Then all the eigenvalues of $H(t)$ can be chosen to be differentiable at $t_0.$
It means that there exist $k$ functions $\lambda_1, \ldots, \lambda_k$ in a neighbourhood of $t_0$ that are differentiable at $t_0$ and $\lambda_1(t), \ldots, \lambda_k(t)$ are the $k$ eigenvalues of $H(t)$ counted with multiplicity.
Further if the map $t\mapsto H(t)$ is $C^1$ on $I,$
then we can choose the eigenvalues to be $C^1$ on the whole of $I.$
See  (\cite{kato}, pp.113-115).
Now since the square root map $A\mapsto A^{1/2}$ is $C^\infty$ on $\mathbb{P}(2n),$ we can directly obtain the following symplectic analogue.

\begin{thm}\label{thm3.1.d}
Let $t\mapsto A(t)$ be a map from an open interval $I$ to $\mathbb{P}(2n)$
that is differentiable at $t_0\in I.$
Then all the symplectic eigenvalues of $A(t)$ can be chosen to be differentiable at $t_0,$
i.e., we can find $n$ functions $\tilde{d}_1,\ldots,\tilde{d}_n$ in a neighbourhood of $t_0$ that are differentiable
at $t_0$ such that $\tilde{d}_1(t),\ldots,\tilde{d}_n(t)$ are the symplectic eigenvalues of $A(t).$
If, in addition, the map $t\mapsto A(t)$ is $C^1$ on $I,$
then $\tilde{d}_1,\ldots,\tilde{d}_n$ can be chosen to be $C^1$ on $I.$
\end{thm}

The continuity of symplectic eigenvector pairs cannot be guaranteed
even if the map $t\mapsto A(t)$ is $C^\infty$ on $I.$
This we show by the following example.
\begin{example}
For each $t\in (-1,1)$ define the $4\times 4$ positive definite matrix $A(t)$ as
$$A(t)=I_2 \otimes P(t)$$

where
$$P(t)=\left[\begin{matrix}
                     \scriptstyle 1- e^{-1/t^2} \cos(2/t) &  \scriptstyle  -e^{-1/t^2} \sin(2/t) \\
                    \scriptstyle - e^{-1/t^2} \sin(2/t) & \scriptstyle 1+e^{-1/t^2} \cos(2/t))\\
                  \end{matrix}\right]$$
									for $t \neq 0$ and $P(0)=I_2.$
									Clearly $t\mapsto A(t)$ is a smooth map.                  
For $t \neq 0,$   $d_1(t)= 1-e^{-1/t^2} $ and $d_2(t)= 1+e^{-1/t^2},$
and $d_1(0)=d_2(0)=1.$
 Let
$ u_1(t)=e_1 \otimes \left[\begin{matrix}
                     \scriptstyle \cos(1/t)  &   \scriptstyle  \sin(1/t) 
                  \end{matrix}\right]^T,   
   v_1(t)=e_2 \otimes \left[\begin{matrix}
                  \scriptstyle  \cos(1/t)  &    \scriptstyle \sin(1/t) 
                  \end{matrix}\right]^T$
                  and
$ u_2(t)=e_1 \otimes \left[\begin{matrix}
                     \scriptstyle \sin(1/t)  &   \scriptstyle  -\cos(1/t) 
                  \end{matrix}\right]^T,   
   v_2(t)=e_2 \otimes \left[\begin{matrix}
                     \scriptstyle \sin(1/t)  &   \scriptstyle  -\cos(1/t) 
                  \end{matrix}\right]^T, $             
  where $e_1, e_2$ are the standard unit vectors in $\mathbb{R}^2.$
  
\sloppy
	One can see that $(u_1(t), v_1(t)) (resp. \ (u_2(t), v_2(t)))$ is a normalised symplectic eigenvector pair
corresponding to $d_1(t) (resp. \ d_2(t)).$ 
Suppose that there exist functions $\tilde{u}, \tilde{v}: (-1,1) \to \mathbb{R}^4,$ continuous at $0$ such that $\left(\tilde{u}(t), \tilde{v}(t) \right)$  forms a normalised symplectic eigenvector pair of $A(t).$
For each $t \neq 0$ the pair $\left(\tilde{u}(t), \tilde{v}(t) \right)$ either corresponds to $d_1(t),$ or to $d_2(t).$
Therefore we can get a sequence $(t_j)_{j \in \mathbb{N}}$ of nonzero terms in $(-1,1)$ converging to $0$
such that for all $j\in\mathbb{N}$ $\left(\tilde{u}(t_j), \tilde{v}(t_j) \right)$ corresponds either to $d_1(t_j)$ or to $d_2(t_j).$
Consider the case when $\left(\tilde{u}(t_j),\tilde{v}(t_j)\right)$  corresponds to $d_1(t_j)$ for all $j.$ 
 For each $j,$ $d_1(t_j)$ is a simple symplectic eigenvalue of $A(t_j).$ 
 This implies that the normalised symplectic eigenvector pair $(\tilde{u}(t_j), \tilde{v}(t_j))$ is of the form
	$	\tilde{u}(t_j) = a_ju_1(t_j)- b_jv_1(t_j), 		\tilde{v}(t_j) = b_ju_1(t_j)+a_jv_1(t_j)$ 
 where $a_j, \ b_j \in \mathbb{R}$ and $a_j^2+b_j^2=1.$ 
 The continuity of $\tilde{u}$ and $\tilde{v}$ at $t=0$ implies that
the limits $ \lim\limits_{j \to \infty} a_j \sin (1/t_j)$ and $\lim\limits_{j \to \infty} b_j \sin (1/t_j)$ exist,
which in turn imply that $\lim\limits_{j \to \infty}  \sin^2 (1/t_j)$ exists.
This is a contradiction. We get a similar contradiction in the other case.
Therefore we conclude that there does not exist any continuous selection of normalised symplectic eigenvector pairs.
\end{example}

However, the symplectic eigenvalues and the symplectic eigenvector pairs can be chosen smoothly under an additional condition as shown in the following theorem.
The proof follows from the smoothness of the map $A\mapsto A^{1/2}$ on $\mathbb{P}(2n),$ and (\cite{klm}, Theorem 7.6).
We say that two functions $f$ and $g$ continuous at $t_0$ meet with infinite order if for every $p\in\mathbb{N}$
there exists a function $h_p$ continuous at $t_0$ such that
$f(t)-g(t)=t^ph_p(t).$ See (3.5) in \cite{klm}.

\begin{thm}
Let $t\mapsto A(t)$ be a smooth map from an open interval $I$ to $\mathbb{P}(2n)$
such that for all $1\le i\ne j\le n$ either $d_i(t)=d_j(t)$ for all $t\in I$
or $d_i(t)$ and $d_j(t)$ do not meet with infinite order at any point in $I.$
Then all the symplectic eigenvalues and corresponding symplectic eigenbasis can be chosen smoothly in $t$ on $I.$
\end{thm}

We now turn to the case when $A(t)$ is a real analytic curve.

\sloppy
\begin{thm}\label{thm_main2}
Let $t\mapsto A(t)$ be a map from an open interval $I$ to $\mathbb{P}(2n)$
that is real analytic at $t_0\in I.$
\begin{itemize}
\item[(i)] If $d$ is a symplectic eigenvalue of $A(t_0)$ with multiplicity $m,$
then for some $\epsilon>0,$ there exist $m$ symplectic eigenvalue maps
$\tilde{d}_1,\ldots,\tilde{d}_m:(t_0-\epsilon,t_0+\epsilon)\to\mathbb{R},$ and $m$
corresponding symplectically orthonormal symplectic eigenvector pair maps
$(\tilde{u}_1,\tilde{v}_1),\ldots,(\tilde{u}_m,\tilde{v}_m):(t_0-\epsilon,t_0+\epsilon)\to\mathbb{R}^{2n}\times \mathbb{R}^{2n}$
that are real analytic at $t_0$ with
each $\tilde{d}_j(t_0)=d.$
\item[(ii)] There exists an $\epsilon>0$ such that all the $n$ symplectic eigenvalues of $A(t)$ and a corresponding symplectic eigenbasis
can be chosen on $(t_0-\epsilon,t_0+\epsilon)$ to be real analytic at $t_0.$
\end{itemize}
\end{thm}

Similar to the case of differentiability, we use the results on analyticity of eigenvalues of Hermitian matrices.
For this we need the following proposition.
Since we could not find an explicit proof of this in literature
we include its proof in the appendix for the convenience of the reader.

\begin{prop}\label{prop71}
Let $t\mapsto A(t)$ be a map from an open interval $I$ to $\mathbb{P}(m)$	
that is real analytic at $t_0\in I.$
Then the composite map $t\mapsto\left(A(t)\right)^{1/2}$ is also real analytic at $t_0.$
\end{prop}

We use the following result for eigenvalues and eigenvectors for Hermitian matrices to prove Theorem \ref{thm_main2}.
See    Kato \cite{kato} and Rellich \cite{rellich}.

\begin{prop}\label{thm8}
Let $t\mapsto H(t)$ be a map from an open interval $I$ to $\mathbb{H}(k)$
that is real analytic at $t_0.$
If $\lambda$ is an eigenvalue of $H(t_0)$ with multiplicity $m,$
then there exists an $\epsilon>0$ so that 
we can find $m$ eigenvalue functions $\lambda_1,\ldots,\lambda_m:(t_0-\epsilon,t_0+\epsilon)\to\mathbb{R}$ and
$m$ corresponding orthonormal eigenvector functions $x_1,\ldots,x_m:(t_0-\epsilon,t_0+\epsilon)\to\mathbb{C}^{k}$ that are real analytic at $t_0.$
Also $\lambda_i(t_0)=\lambda$ for all $i=1,\ldots,m.$
\end{prop}

\noindent{\it\bf{\it Proof of Theorem \ref{thm_main2}}}:
Let $H(t)$ be the Hermitian matrix $\imath A^{1/2}(t)JA^{1/2}(t).$
Since $t\mapsto A(t)$ is real analytic at $t_0,$ by Proposition \ref{prop71}, the map $t\mapsto H(t)$ is also real analytic at $t_0.$
By Proposition \ref{prop1sev}, the multiplicity of the eigenvalue $d$ of $H(t_0)$ is $m.$
Hence by Proposition \ref{thm8}, there exist an $\epsilon>0,$
and $m$ functions $\tilde{d}_1,\ldots, \tilde{d}_m:(t_0-\epsilon,t_0+\epsilon)\to\mathbb{R}$ and
$m$ functions $x_1,\ldots,x_m:(t_0-\epsilon,t_0+\epsilon)\to\mathbb{C}^{2n}$ that are real analytic at $t_0$ such that
$\tilde{d}_1(t), \ldots, \tilde{d}_m(t)$ are $m$ eigenvalues of $H(t)$
and $\{x_1(t), x_2(t),\ldots, x_m(t)\}$ is a corresponding orthonormal set of eigenvectors.
Also $\tilde{d}_j(t_0)=d$ for all $j=1,\ldots,m.$
Since $H(t)$ is invertible for every $t$ and $d>0,$
 each $\tilde{d}_j(t)>0.$
Hence $\tilde{d}_j(t)$ is a symplectic eigenvalue of $A(t)$ for every $t\in (t_0-\epsilon,t_0+\epsilon)$ and $j=1,\ldots,m.$
Let $x_j(t)=\overline{u}_j(t)-\imath \overline{v}_j(t)$ be the Cartesian decomposition of $x_j(t).$
For every $t\in (t_0-\epsilon,t_0+\epsilon)$ let $\tilde{u}_j(t)=\sqrt{2\tilde{d}_j(t)}A^{-1/2}(t)\overline{u}_j(t)$ and $\tilde{v}_j(t)=\sqrt{2\tilde{d}_j(t)}A^{-1/2}(t)\overline{v}_j(t).$
Since $\tilde{d}_j(t)$ and $A^{-1/2}(t)$ are real analytic at $t_0,$ $\tilde{u}_j(t)$ and $\tilde{v}_j(t)$ are real analytic at $t_0.$
Finally by Proposition \ref{prop1sev}, $\{(\tilde{u}_j(t),\tilde{v}_j(t)):j=1,\ldots,m\}$ is a symplectically orthonormal set of symplectic eigenvector pairs of $A(t)$
corresponding to $\tilde{d}_1(t),\ldots,\tilde{d}_m(t).$
This proves (i).

Let $\overline{d}_1<\cdots<\overline{d}_k$ be distinct symplectic eigenvalues of $A(t_0)$ with multiplicities $m_1,\ldots,m_k,$ respectively.
By statement (i) of the theorem, we can find an $\epsilon>0$ and $n$
symplectic eigenvalue functions
$\tilde{d}_{1,1}(t),\ldots,\tilde{d}_{1,m_1}(t),\ldots,\tilde{d}_{k,1}(t),\ldots,\tilde{d}_{k,m_k}(t)$ of $A(t)$ on $(t_0-\epsilon,t_0+\epsilon)$
that are real analytic at $t_0.$
Also for each $j=1,\ldots,k,$ we can choose corresponding symplectically orthonormal symplectic eigenvector pairs $(\tilde{u}_{j,i}(t),\tilde{v}_{j,i}(t)),$ $1\le i\le m_j,$
that are real analytic at $t_0.$Using Proposition \ref{prop1b}, we can assume that $\epsilon>0$ is small enough so that
for all $t\in (t_0-\epsilon,t_0+\epsilon)$ $\tilde{d}_{r,i}(t)\ne\tilde{d}_{s,j}(t)$ for all $1\le i\le m_r$ and $1\le j\le m_s,$ $r\ne s.$
Thus by Corollary \ref{cor3sev}
the symplectic eigenvector pairs $(\tilde{u}_{j,i}(t),\tilde{v}_{j,i}(t)),$ $1\le i\le m_j,$ $1\le j\le k,$ form the required symplectic eigenbasis.
\qed
\vskip.1in

By arguing in a similar way as in the proof of Theorem \ref{thm_main2}(i) and
using the analogous result for eigenvalues and eigenvectors of Hermitian matrices (see \cite{kato}, Ch.II, Sec.6), we can obtain the following theorem.
\begin{thm}\label{thm_global}
Let $t\mapsto A(t)$ be a real analytic map from an open interval $I$ to $\mathbb{P}(2n).$
Then we can choose $n$ symplectic eigenvalue functions
and corresponding symplectic eigenbasis map such that they are real analytic on $I.$
\end{thm}

We have seen that the ordered tuple $d_1\le d_2\le\cdots\le d_n$ of symplectic eigenvalues need not be differentiable when the multiplicities are greater than one.
But they can be proved to be piecewise real analytic if the map $t\mapsto A(t)$ is real analytic on $I.$

\begin{thm}\label{thm3.2.pa}
Let $t\mapsto A(t)$ be a real analytic map from an open interval $I$ to $\mathbb{P}(2n),$
and let $[a,b]$ be any compact interval contained in $I.$
Then for each $j=1,\ldots,n,$ the map $t\mapsto d_j(t)=d_j(A(t))$ is piecewise real analytic on $[a,b].$
Further for each $t\in  [a,b],$ we can find a symplectic eigenbasis
$\{u_1(t),\ldots,u_n(t),v_1(t),\ldots,v_n(t)\}$ of $A(t)$ corresponding to $d_1(t),\ldots,d_n(t)$ such that the maps
$u_1,\ldots,u_n,v_1,\ldots,v_n$ are also piecewise real analytic on $[a,b].$
\end{thm}

\begin{proof}
By Theorem \ref{thm_global}, we can find
$n$ symplectic eigenvalues $\tilde{d}_1(t),\ldots,\tilde{d}_n(t)$ of $A(t)$ and a corresponding symplectic eigenbasis
$\{\tilde{u}_1(t),\ldots,\tilde{u}_n(t),\tilde{v}_1(t),\ldots,\tilde{v}_n(t)\}$ such that each of the maps
$\tilde{d}_j,\tilde{u}_j$ and $\tilde{v}_j$ are real analytic on $I.$

Define $\mathcal{I}$ to be the set of all ordered pairs $(i,j),$ $1\le i\ne j\le n,$ such that
$\tilde{d}_i(t)\ne \tilde{d}_j(t)$ for at least one $t$ in $[a,b].$
Let $E$ be the set of all points $t$ in $[a,b]$ such that $\tilde{d}_i(t)=\tilde{d}_j(t)$ for some $(i,j)\in\mathcal{I}.$
By using the real analyticity of the maps $\tilde{d}_1,\ldots,\tilde{d}_n$ and the definition of the set $\mathcal{I},$
we can see that $E$ is finite.
Then for every $i=1,\ldots,n$ the multiplicity of
 $\tilde{d}_i(t)$ is the same for all $t$ in $[a,b]\setminus E.$
Hence $\tilde{d}_1,\ldots,\tilde{d}_n$ can be reordered so that
$\tilde{d}_i(t)=d_i(t)$ for all $t\in [a,b]\setminus E.$
The theorem thus follows by suitably reordering the symplectic eigenvalues $\tilde{d}_1,\ldots,\tilde{d}_n$
and correspondingly reordering the symplectic eigenvalue pairs $\tilde{u}_1,\ldots,\tilde{u}_n,\tilde{v}_1,\ldots,\tilde{v}_n.$
\end{proof}

\section{Symplectic analogue of Lidskii's theorem and other applications}
The main object of this section is to derive a Lidskii type result, that is, a
majorisation inequality between the symplectic eigenvalues of the sum of two positive definite matrices and those of the two matrices.
We start this section by introducing a notion of symplectic projections that is of independent interest
and is useful in the proof of our main theorem of this section.

Let $S=\{x_1,\ldots,x_k,y_1,\ldots,y_k\}$ be a symplectically orthonormal subset of $\mathbb{R}^{2n}.$
Define the map $P_S$ on $\mathbb{R}^{2n}$ as
\begin{equation}
P_S(x)=\sum\limits_{i=1}^{k}\left(\langle x,Jy_i\rangle Jy_i+\langle x,Jx_i\rangle Jx_i\right).\label{eqflr1}
\end{equation}
Suppose $M$ is the $2n\times 2k$ matrix
\begin{equation}
M=\begin{bmatrix}Jx_1, \ldots, Jx_k, Jy_1, \ldots, Jy_k\end{bmatrix}.\label{eqflr2}
\end{equation}
It is easy to see that $P_S=MM^T,$ and so, it is a positive semidefinite matrix.
In fact, it can be seen that the kernel of $P_S$ is the symplectic complement of $S,$ and
hence $P_S\in\mathbb{P}_s(2n)$ with symplectic eigenvalues $1$ and $0$ with multiplicities $k$ and $n-k,$ respectively.
We call $P_S$ to be the {\it symplectic projection} associated with the set $S.$
If $k=n,$ i.e., $S$ is a symplectic basis of $\mathbb{R}^{2n},$ then $P_S$ is a positive definite symplectic matrix with all its symplectic eigenvalues $1.$

\begin{prop}\label{propflr1}
Let $S=\{u_1,\ldots,u_k,v_1,\ldots,v_k\}$ and $T=\{x_1,\ldots,x_m,y_1,\ldots,y_m\}$
be two symplectically orthonormal subsets of $\mathbb{R}^{2n},$
and let $P$ and $Q$ be the symplectic projections associated with them.
Let $M$ and $N$ be the $2n\times 2k$  and $2n\times 2m$ matrices given by \eqref{eqflr2} corresponding to the sets $S$ and $T,$ respectively.
Then $P=Q$ if and only if $k=m$ and $M=NU$ for some $2k\times 2k$ orthosymplectic (symplectic as well as orthogonal) matrix $U.$
\end{prop}

\begin{proof}
If $k=m$ and $M=NU,$ the equality $P=Q$ easily follows from the orthogonality of $U,$
and the fact that $P=MM^T$ and $Q=NN^T.$

Conversely, let $P=Q.$
Clearly the subspaces spanned by $S$ and $T$ are the same, and hence $k=m.$
By \eqref{eqflr1}
$$Px_j=\sum\limits_{i=1}^{k}\left(\alpha_{ij}Ju_i+\beta_{ij} Jv_i\right)$$
for all $j=1,\ldots,k.$
Here $\alpha_{ij}=\langle x_j,Ju_i\rangle$ and $\beta_{ij}=\langle x_j,Jv_i\rangle,$ $1\le i,j\le k.$
Since $P=Q,$ $Px_i=Jy_i.$
This gives
\begin{equation}
y_j=\sum\limits_{i=1}^{k}\left(\alpha_{ij}u_i+\beta_{ij}v_i\right).\label{eqflr3}
\end{equation}
Also since $x_j$ belongs to the span of the symplectically orthonormal vectors $u_1,\ldots,u_k,v_1,\ldots,v_k,$
\begin{eqnarray}
x_j &=&\sum\limits_{i=1}^{k}\left(\langle x_j,Jv_i\rangle u_i-\langle x_j,Ju_i\rangle v_i\right)\\
&=&\sum\limits_{i=1}^{k}\left(\beta_{ij}u_i-\alpha_{ij}v_i\right).\label{eqflr4}
\end{eqnarray}
Let $X$ and $Y$ be the $k\times k$ matrices $X=\begin{bmatrix}\alpha_{ij}\end{bmatrix}$ and $Y=\begin{bmatrix}\beta_{ij}\end{bmatrix},$
and $U$ be the $2k\times 2k$ matrix
$$U=\begin{bmatrix}Y & X\\
-X & Y\end{bmatrix}.$$
Using the fact that $x_1,\ldots,x_k,y_1,\ldots,y_k$ are symplectically orthonormal,
we can see that the columns of $U$ are orthonormal as well as symplectically orthonormal vectors in $\mathbb{R}^{2k}.$
Finally, from \eqref{eqflr3} and \eqref{eqflr4} we obtain
$N=MU.$
\end{proof}

We now give an equivalent statement for Williamson's theorem in terms of symplectic projections.

\begin{prop}\label{propflr2}
For every $B$ in $\mathbb{P}(2n)$ there exist distinct positive numbers $\mu_1,\ldots,\mu_m$ and
symplectic projections $P_1,\ldots,P_m$  that satisfy the following conditions.
\begin{itemize}
\item[(i)] $P_jJP_k=0$ for all $j\ne k,$ $j,k=1,\ldots,m.$
\item[(ii)] $\sum\limits_{k=1}^{m}P_kJP_k=J.$
\item[(iii)] $B=\sum\limits_{k=1}^{m}\mu_kP_k.$
\end{itemize}
The numbers $\mu_1,\ldots,\mu_m$ and the symplectic projections $P_1,\ldots,P_m$
are uniquely determined by the above conditions.
Further, for every $1\le j\le m,$ $\mu_j$ is a symplectic eigenvalue of $B$ and $P_j$ is the symplectic projection associated with a
symplectically orthonormal set of eigenvector pairs of $B$ corresponding to $\mu_j.$
\end{prop}

\begin{proof}
Let $\mu_1,\ldots,\mu_m$ be the distinct symplectic eigenvalues of $B$
with multiplicities $k_1,\ldots,k_m,$ respectively.
For every $j=1,\ldots,m$
let $S_j=\{u_{j,1},\ldots,u_{j,k_j},v_{j,1},\ldots,v_{j,k_j}\}$
be a symplectically orthonormal set of symplectic eigenvector pairs of $B$ corresponding to $\mu_j.$
Let $P_j$ be the symplectic projection associated with $S_j.$
By the definition of symplectic projections and Williamson's theorem, we can see that
$\mu_1,\ldots,\mu_m$ and $P_1,\ldots,P_m$ satisfy (i)-(iii).

Now, let $\eta_1,\ldots,\eta_l$ be $l$ distinct positive numbers and $Q_1,\ldots,Q_l$ be symplectic projections that also satisfy (i)-(iii).
For every $j=1,\ldots,l,$ let $T_j=\{x_{j,1},\ldots,x_{j,r_j},y_{j,1},\ldots,y_{j,r_j}\}$
be a symplectically orthonormal set corresponding to $Q_j.$
By using (i) and (iii), we can see that each $\eta_j$ is a symplectic eigenvalue of $B,$
and $(x_{j,i},y_{j,i}),$ $1\le i\le r_j,$ are the symplectically orthonormal symplectic eigenvector pairs corresponding to $\eta_j.$
Condition (ii) implies that $\{\eta_1,\ldots,\eta_l\}$ forms the set of all distinct symplectic eigenvalues of $B.$
By the uniqueness of symplectic eigenvalues, we have $l=m$ and $\{\mu_1,\ldots,\mu_m\}=\{\eta_1,\ldots,\eta_l\}.$
We can assume that $\mu_j=\eta_j$ for all $j=1,\ldots,m.$
By (iii) we see that $r_j$ is equal to the multiplicity of $\mu_j.$
Since symplectic eigenvector pairs corresponding to different eigenvalues are symplectically orthogonal, $S_j$ is symplectically orthogonal to $T_k$ for all $j\ne k.$
Consequently $P_jx=0$ for all $x\in T_k$ and for all $k\ne j.$
Thus for every $(x_{j,i},y_{j,i})$ in $T_j$ we have
$$\mu_jQ_jx_{j,i}=\mu_jJy_{j,i}=Bx_{j,i}=\mu_jP_jx_{j,i}.$$
and since $\mu_j\ne 0,$ $P_jx_{j,i}=Q_jx_{j,i}.$
Similarly $P_jy_{j,i}=Q_jy_{j,i}.$
Since $\cup T_j$ forms a basis for $\mathbb{R}^{2n},$ we get $P_j=Q_j$ for all $j=1,\ldots,m.$
\end{proof}

By using Proposition \ref{propflr1} and the uniqueness of symplectic projections in Proposition \ref{propflr2},
we get the following:

\begin{cor}
Let $A\in\mathbb{P}(2n),$
and let $d$ be its symplectic eigenvalue with multiplicity $m.$
Let $S=\{u_1,\ldots,u_m,v_1,\ldots,v_m\}$
be a symplectically orthonormal set of symplectic eigenvector pairs of $A$ corresponding to $d.$
Then the set $T=\{x_1,\ldots,x_m,y_1,\ldots,y_m\}$
is also a symplectically orthonormal set of symplectic eigenvector pairs corresponding to $d$
if and only if there exists a $2m\times 2m$ orthosymplectic matrix $U$
such that
$$N=MU,$$
where $M$ and $N$ are $2n\times 2m$ matrices with columns $u_1,\ldots,u_m,v_1,\ldots,v_m$
and $x_1,\ldots,x_m,y_1,\ldots,y_m,$ respectively.
\end{cor}

%

We can also verify that if $d_1(B),\ldots,d_n(B)$ are the symplectic eigenvalues of
$B$ and
$\{u_1,\ldots,u_n,v_1,\ldots,v_n\}$ is a corresponding symplectic eigenbasis, then
$$B=\sum\limits_{j=1}^{n}d_j(B)P_j,$$
where $P_j$ is the symplectic projection corresponding to $\{u_j,v_j\}.$
\vskip.1in

For a real vector $x=(x_1,\ldots,x_n),$ we denote by $x^\uparrow$ the vector $(x_1^\uparrow,\ldots,x_n^\uparrow)$ obtained by rearranging the components of $x$ in increasing order, i.e.,
$$x_1^\uparrow\le \cdots\le x_n^\uparrow.$$
We say $x$ is {\it supermajorised} by $y,$ in symbols $x \prec^w y,$ if for $1 
\leq k \leq n$ 
\begin{equation}
 \sum_{j=1}^{k} \, x_j^{\uparrow} \ge \sum_{j=1}^{k} \,y_j^{\uparrow}. 
\label{eqsup}
\end{equation}
We say that $x$ {\it majorises} $y$ (or $y$ {\it is majorised by} $x$)
if the two sides in the above inequalities are equal when $k=n.$

An $n\times n$ matrix $B=\begin{bmatrix}b_{ij}\end{bmatrix}$ is called {\it doubly superstochastic} if there exists an
$n\times n$ doubly stochastic matrix $A=\begin{bmatrix}a_{ij}\end{bmatrix}$ such that $b_{ij}\ge a_{ij}$ for all $i,j=1,\ldots,n.$
See \cite{ando}. It can be seen that the set of doubly superstochastic is a closed and convex subset of $\mathbb{M}(n).$ 
In order to prove Theorem \ref{thm1l}, we will use the following fundamental result in the theory of  majorisation.

\begin{lem}\label{lem1l}
The following two conditions are equivalent:
\begin{itemize}
\item[(i)] An $n \times n$ matrix $A$ is doubly superstochastic.
\item[(ii)] $Ax \prec^{w} x$ for every positive $n$-vector $x.$
\end{itemize}
\end{lem}

For a positive definite matrix $A,$ we denote by $d^\uparrow(A)$ the $n$-tuple of symplectic eigenvalues arranged in increasing order, i.e.,
$$d^\uparrow(A)=(d_1(A),\ldots,d_n(A)).$$

\begin{thm}\label{thm1l}
Let $A,B$ be two $2n\times 2n$ positive definite matrices.
Then
\begin{equation}
d^\uparrow(A+B)-d^\uparrow(A)\prec^w d^\uparrow(B).\label{eq1l}
\end{equation}
\end{thm}

\begin{proof}
Define the map $\varphi:[0,1]\to\mathbb{P}(2n)$ as
$$\varphi(t)=A+tB.$$
Clearly $\varphi$ is real analytic with $\varphi^\prime(t)=B.$
Let $1\le j\le n,$ and let $d_j(t)=d_j(\varphi(t)).$
By Theorem \ref{thm3.2.pa}, $d_j$ is piecewise real analytic.
Also by the same theorem, we can find a piecewise real analytic symplectic eigenbasis
$\beta(t)=\{u_1(t),\ldots,u_n(t),v_1(t),\ldots,v_n(t)\}$ of $\varphi(t)$ corresponding to $d_1(t),\ldots,d_n(t).$
For any $t$ in $[0,1]$ at which $d_j,$ $u_j$ and $v_j$ are real analytic, we have
\begin{equation}
d_j^\prime(t)=\frac{1}{2}\left(\langle u_j(t),Bu_j(t)\rangle+\langle v_j(t),Bv_j(t)\rangle\right).\label{eqll}
\end{equation}
Let $\mu_1 \leq \ldots \leq \mu_n$  be the symplectic eigenvalues of $B$ and $\beta=\{x_1,\ldots,x_n,y_1,\ldots,y_n\}$ be a corresponding symplectic eigenbasis.
Let $P_j$ be the symplectic projection corresponding to $(x_j,y_j).$
Then $B=\sum\limits_{j=1}^{n}\mu_jP_j.$
Thus by using this expression for $B$ and using \eqref{eqflr1} for $P_k$ in \eqref{eqll}, we get
\begin{eqnarray}
d_j^\prime(t)&=&\sum\limits_{k=1}^{n}\frac{\mu_k}{2}\left(\langle u_j(t),P_ku_j(t)\rangle+\langle v_j(t),P_kv_j(t)\rangle\right)\nonumber\\
&=&\sum\limits_{k=1}^{n}\frac{\mu_k}{2}\left(\langle u_j(t),Jy_k\rangle^2+\langle u_j(t),Jx_k\rangle^2\right.\nonumber\\
& & \ \left. +\langle v_j(t),Jy_k\rangle^2+\langle v_j(t),Jx_k\rangle^2\right).\label{eq2l}
\end{eqnarray}
Since $\beta(t)$ and $\beta$ are symplectic bases of $\mathbb{P}(2n),$ the matrix $M(t)$ with $rs\textrm{th}$ entry 
$$m_{rs}(t)=\begin{cases}
\langle u_j(t),Jx_k\rangle & r=j,s=k,1\le j,k\le n\\
\langle u_j(t),Jy_k\rangle & r=j,s=n+k,1\le j,k\le n\\
\langle v_j(t),Jx_k\rangle & r=n+j,s=k,1\le j,k\le n\\
\langle v_j(t),Jy_k\rangle & r=n+j,s=n+k,1\le j,k\le n
\end{cases}$$
is a symplectic matrix.
Let $\widetilde{M}(t)$ be the $n\times n$ matrix with $jk\textrm{th}$ entry
$$\frac{m_{jk}^2(t)+m_{j(n+k)}^2(t)+m_{(n+j)k}^2(t)+m_{(n+j)(n+k)}^2(t)}{2}.$$
Then by \eqref{eq2l}, we see that $d_j^\prime(t)$ is the $j\textrm{th}$ component of the vector $\widetilde{M}(t)d^\uparrow(B),$
i.e., 
\begin{equation}
d^{\prime}(t)=\widetilde{M}(t)d^\uparrow(B).\label{eq3l}
\end{equation}
where $d^{\prime}(t)=(d_1^{\prime}(t), \ldots, d_n^{\prime}(t))^T.$ Since $d_j,u_j,v_j$ are piecewise real analytic on $[0,1],$
the maps $d_j$ and $\widetilde{M}$ are integrable on $[0,1].$
Denote by $\overline{M},$ the $n\times n$ matrix
$$\overline{M}=\int\limits_{0}^{1}\widetilde{M}(t)\d t.$$
By  (\cite{bj}, Theorem 6) each $\widetilde{M}(t)$ is doubly superstochastic.
Since the set of doubly superstochastic matrices is closed and convex, $\overline{M}$ is also doubly superstochastic.
Integrating \eqref{eq3l}, we get
$$d^\uparrow(A+B)-d^\uparrow(A)=\overline{M}d^\uparrow(B).$$
We finally obtain \eqref{eq1l} by Lemma \ref{lem1l}.
\end{proof}
\begin{cor}
For $A,B\in\mathbb{P}(2n),$ and for all $1\le i_1<\cdots<i_k\le n,$
\begin{equation}
\sum\limits_{j=1}^{k}d_{i_j}(A+B)\ge\sum\limits_{j=1}^{k}d_{i_j}(A)+\sum\limits_{j=1}^{k}d_{j}(B).\label{eq1m}
\end{equation}
In particular,
\begin{equation}
d_j(A+B)\ge d_j(A)+d_1(B),\label{eqflr5}
\end{equation}
and
$$d_j(A+I)\ge d_j(A)+1.$$
Here $I$ denotes the $2n\times 2n$ identity matrix.
\end{cor}

When $\{i_1,\ldots,i_k\}$ is the set $\{1,\ldots,k\}$ in \eqref{eq1m},
we obtain the inequalities first proved by Hiroshima. See \cite{bj,h}.
The inequalities \eqref{eqflr5} were proved recently by R. Bhatia in \cite{rbhs} in the case when $A$ and $B$ are of some specific form.

We also point out that the supermajorisation in \eqref{eq1l} cannot be replaced by majorisation.
Let $A=\begin{bmatrix}2 & 1\\1 & 2\end{bmatrix}$ and $B=I_2,$ the $2\times 2$ identity matrix.
The only symplectic eigenvalues of $A,B$ and $A+B$ are
$$d_1(A)=\sqrt{3},\, d_1(B)=1\textrm{ and }d_1(A+B)=2\sqrt{2}.$$
Clearly $d_1(A+B)>d_1(A)+d_1(B).$
\vskip.1in
Following is a simple application of Theorem \ref{thm1l}.

\begin{cor}
For all $k=1,\ldots,n$ and  $1 \leq i_1<\cdots< i_k \leq n,$ the map $A \mapsto \sum\limits_{j=1}^{k}d_{i_j}(A)$ on $\mathbb{P}(2n)$
 has neither a local minimiser nor a local maximiser in $\mathbb{P}(2n).$
In particular, for every $j=1,\ldots,n,$ the map $A\mapsto d_j(A)$ has neither a local minimiser nor a local maximiser in $\mathbb{P}(2n).$
\end{cor}
\begin{proof}
Let $I$ denote the $2n \times 2n$ identity matrix.
Let $A \in \mathbb{P}(2n)$ and $\epsilon > 0$ be such that
$A \pm \epsilon I \in \mathbb{P}(2n).$
Then replacing $B$ by $ \epsilon I$ in \eqref{eq1m} we get 
\begin{equation*}
 \sum\limits_{j=1}^{k}d_{i_j}(A+ \epsilon I)\ge\sum\limits_{j=1}^{k}d_{i_j}(A)+k \epsilon
\end{equation*}
Similarly, replacing $A$ by $A - \epsilon I$ and $B$ by $\epsilon I,$ we get
\begin{equation*}
 \sum\limits_{j=1}^{k}d_{i_j}(A)\ge\sum\limits_{j=1}^{k}d_{i_j}(A- \epsilon I)+k \epsilon
\end{equation*}
Consequently, we get 
\begin{align*}
\sum\limits_{j=1}^{k}d_{i_j}(A+ \epsilon I) > \sum\limits_{j=1}^{k}d_{i_j}(A) > \sum\limits_{j=1}^{k}d_{i_j}(A- \epsilon I)
\end{align*}
\end{proof}

A $2n\times 2n$ real positive definite matrix $A$ is a {\it covariance matrix} corresponding to a Gaussian state (or a Gaussian covariance matrix)if and only if it satisfies 
$$A+\frac{\imath }{2}J\ge 0.$$
This is equivalent to saying that all the symplectic eigenvalues $d_j(A)\ge 1/2.$
The von Neumann entropy of a Gaussian state with covariance matrix $A$ is given by\begin{equation}
S(A)=\sum\limits_{i=1}^{n}\left[\left(d_i+\frac{1}{2}\right)\log\left(d_i+\frac{1}{2}\right)-\left(d_i-\frac{1}{2}\right)\log\left(d_i-\frac{1}{2}\right)\right].\label{eq1g}
\end{equation}

\begin{thm}
Let $t\mapsto A(t)$ be a real analytic map from an open interval $I$ to the set of Gaussian covariance matrices.
Then the entropy map $S(t)=S(A(t))$ is monotonically increasing (decreasing) on $I$ if $A^\prime(t)$ is positive (negative) semidefinite for all $t$ in $I.$
\end{thm}

\begin{proof}
Since $t\mapsto A(t)$ is real analytic on $I,$
by Theorem \ref{thm_global}, we can choose the symplectic eigenvalues $\tilde{d}_1(t),\ldots,\tilde{d}_n(t),$
and a corresponding symplectic eigenbasis
$\{\tilde{u}_1(t),\ldots,\tilde{u}_n(t),\tilde{v}_1(t),\ldots,\tilde{v}_n(t)\}$ of $A(t)$ to be real analytic on $I.$
By Remark \ref{remflr}, we have
$$\tilde{d}_j^\prime(t)=\frac{1}{2}\left(\langle \tilde{u}_j(t),A^\prime(t)\tilde{u}_j(t)\rangle+\langle\tilde{v}_j(t),A^\prime(t)\tilde{v}_j(t)\rangle\right).$$
If $A^\prime(t)$ is positive semidefinite, then each $\tilde{d}_j^\prime(t)\ge 0.$
Since the maps $\tilde{d}_j$ are continuous and $S$ is a continuous map of $\tilde{d}_j,$
$t\to S(t)$ is continuous on $I.$
The matrices $A(t)$ are Gaussian covariance matrices for all $t.$
Hence $\tilde{d}_j(t)\ge 1/2$ for all $1\le j\le n$ and for all $t\in I.$
Let $F$ be the set $\{i:\tilde{d}_i(t)=1/2\textrm{ for all }t\in I\}.$
If  $F = \{1,\ldots,n\},$ then $S(t)=0$ for all $t \in I.$
So, let $F\ne\{1,\ldots,n\}.$
 Let $I_0\subseteq I$ be any open bounded interval.
Clearly it suffices to show that $S(t)$ is monotonically increasing on $I_0.$
Consider the set $E=\{t\in I_0:\tilde{d}_j(t)=1/2, 1\le j\le n, j \notin F\}.$
By the analyticity of $\tilde{d}_j,$ we know that $E$ is finite.
For all $t\in I_0\setminus E,$ we have
$$S^\prime(t)=\sum_{\substack{1 \le j \le n \\ j \notin F}}\log\left(\frac{2\tilde{d}_j(t)+1}{2\tilde{d}_j(t)-1}\right)\tilde{d}_j^\prime(t),$$
Hence $S^\prime(t)\ge 0$ if $A^\prime(t)\ge 0$ for all $t\in I_0\setminus E.$
The above fact together with the continuity of $S(t)$ proves the theorem.
\end{proof}

For a matrix $A$ we denote by $\kappa(A)$ the condition number of $A,$ i.e. $\kappa(A)=\|A\|\|A^{-1}\|.$
In our final result we give a perturbation bound for symplectic eigenvalues.
Different perturbation bounds have been given in \cite{bj} and \cite{idel} using very different techniques than ours. 

\begin{thm}
Let $A,B\in\mathbb{P}(2n).$ Then
\begin{equation}
\max\limits_{1\le j\le n}|d_j(A)-d_j(B)|\le K(A,B)\|A-B\|,\label{eq1p}
\end{equation} 
where $K(A,B)=\int\limits_{0}^{1}\kappa(A+t(B-A)) \d t.$
\end{thm}

\begin{proof}
Define $\varphi:[0,1]\to\mathbb{P}(2n)$ as
$$\varphi(t)=A+t(B-A).$$
As in the proof of Theorem \ref{thm1l}, we see that $d_j(t)=d_j(\varphi(t))$ is piecewise real analytic on $[0,1],$
and we can choose a corresponding piecewise real analytic symplectic eigenbasis $\beta(t)=\{u_1(t),\ldots,u_n(t),v_1(t),\ldots,v_n(t)\}.$
Then for $t$ where $d_j,u_j,v_j$ are real analytic, we have
$$d_j^\prime(t)=\frac{1}{2}\left(\langle u_j(t),(B-A)u_j(t)\rangle+\langle v_j(t),(B-A)v_j(t)\rangle\right).$$
Integrating the above equation, we get
\begin{align}
& |d_j(B)-d_j(A)|\nonumber\\
& = |\int\limits_{0}^{1}d_j^\prime(t)\d t|\nonumber\\
& \le \frac{1}{2}\int\limits_{0}^{1}|\langle u_j(t),(B-A)u_j(t)\rangle+\langle v_j(t),(B-A)v_j(t)\rangle|\d t\nonumber\\
& \le \frac{1}{2}\int\limits_{0}^{1}\left(\|u_j(t)\|^2+\|v_j(t)\|^2\right)\d t\, \|A-B\|.\label{eq2p}
\end{align}
Since $(u_j(t),v_j(t))$ is a normalised symplectic eigenvector pair of $\varphi(t)$ corresponding to $d_j(t),$
\begin{eqnarray*}
\|u_j(t)\|^2+\|v_j(t)\|^2 &\le & \|\varphi(t)^{-1}\|\left(\|\varphi(t)^{1/2}u_j(t)\|^2+\|\varphi(t)^{1/2}v_j(t)\|^2\right)\\
&= & \|\varphi(t)^{-1}\|2d_j(t)\le 2\kappa(\varphi(t)).
\end{eqnarray*}
Thus \eqref{eq2p} gives \eqref{eq1p}.
\end{proof}


\section*{Appendix: Proof of Proposition \ref{prop71}}

\begin{lem}\label{lema1}
Let $\mathcal{X}$ and $\mathcal{Y}$ be Banach spaces,
and let $T:\mathcal{X}^k\to\mathcal{Y}$ be a bounded $k$-linear map.
Suppose $\sum\limits_{n=0}^{\infty}a_{jn}$ is an absolutely convergent series in $\mathcal{X}$ with sum $a_j$ for all $j=1,\ldots,k.$
For each $n,$ let $c_n=\sum\limits_{j_1+\cdots+j_k=n}T(a_{1{j_1}},\ldots,a_{k{j_k}}).$
Then the series $\sum\limits_{n=0}^{\infty}c_n$ is absolutely convergent in $\mathcal{Y}$ and has sum $T(a_1,\ldots,a_k).$
\end{lem}

\begin{proof}
The absolute convergence of the series $\sum\limits_{n=0}^{\infty}c_n$ follows from Merten's theorem for Cauchy products of series of real numbers.
We shall prove that its sum is $T(a_1,\ldots,a_k)$ by induction on $k.$
When $k=1,$ the statement directly follows from the boundedness and linearity of $T.$
Assume that the result holds for $k.$
Let $\sum\limits_{n=0}^{\infty}a_{jn}$ $(1\le j\le k)$ and $\sum\limits_{n=0}^{\infty}b_n$ be absolutely convergent series in $\mathcal{X}$ such that
$a_j=\sum\limits_{n=0}^{\infty}a_{jn}$ and $b=\sum\limits_{n=0}^{\infty}b_n.$
 
For each $m,$ define the map $\tilde{T}_m$ from $\mathcal{X}\to\mathcal{Y}$ as
$$\tilde{T}_m(x)=\sum\limits_{{j_1}+\cdots+{j_k}=m}T(a_{1{j_1}},\ldots,a_{k{j_k}},x).$$
It is easy to see that $\tilde{T}_m$ is linear and bounded with $\|\tilde{T}_m\|\le \|T\|\sum\limits_{{j_1}+\cdots+{j_k}=m}\|a_{1{j_1}}\|\cdots\|a_{k{j_k}}\|.$
Since each $\sum\limits_{n=0}^{\infty}\|a_{jn}\|$ is convergent, by Merten's theorem for Cauchy products of  series of real numbers,
 we see that $\sum\limits_{m=0}^{\infty}\|\tilde{T}_m\|$ converges.
Let $K=\sum\limits_{m=0}^{\infty}\|\tilde{T}_m\|.$
For each $j\ge 0,$ let
$$x_j=\tilde{T}_j(b),$$
and
$$c_j=\sum\limits_{l=0}^{j}\tilde{T}_{j-l}(b_{l}).$$
Clearly $c_j=\sum\limits_{{j_1}+\cdots+{j_k}+l=j}T(a_{1{j_1}},\ldots,a_{k{j_k}},b_l).$
We need to show that $\sum\limits_{j=0}^{\infty}c_j$ is convergent to $T(a_1,\ldots,a_k,b).$
Let $(X_n),$ $(C_n)$ and $(B_n)$ be the sequences of partial sums
of the series $\sum\limits_{j=0}^{\infty}x_j,$ $\sum\limits_{j=0}^{\infty}c_j$ and $\sum\limits_{j=0}^{\infty}b_j,$ respectively.
By induction hypothesis, $\sum\limits_{j=0}^{\infty}x_j$ is absolutely convergent and its sum equals $T(a_1,\ldots,a_j,b).$
Take $d_n=b-B_n$ and $E_n=\sum\limits_{j=0}^{n}\tilde{T}_j(d_{n-j}).$
We have
\begin{eqnarray*}
C_n&=& \sum\limits_{j=0}^{n}\sum\limits_{l=0}^{j}\tilde{T}_l(b_{j-l})\\
&=& \sum\limits_{l=0}^{n}\sum\limits_{j=l}^{n}\tilde{T}_l(b_{j-l})\\
&=& \sum\limits_{l=0}^{n}\tilde{T}_l\left(\sum\limits_{j=0}^{n-l}b_j\right)=\sum\limits_{l=0}^{n}\tilde{T}_l(B_{n-l})\\
&=&\sum\limits_{l=0}^{n}\tilde{T}_l(b)-\sum\limits_{l=0}^{n}\tilde{T}_l(d_{n-l})\\
&=&X_n-E_n.
\end{eqnarray*}
It suffices to show that $E_n\to 0$ as $n\to\infty.$
Since $d_n\to 0,$ we can find a positive number $M$ such that $\|d_n\|\le M$ for all $n\ge 0.$
Given an $\epsilon>0,$ choose $N$ in $\mathbb{N}$ such that for all $n\ge N$
$$\|d_n\|<\frac{\epsilon}{2(K+1)}$$
and
$$\sum\limits_{j=n+1}^{\infty}\|\tilde{T}_j\|<\frac{\epsilon}{2M}.$$
Then for all $n>2N$ we can write
\begin{eqnarray*}
\|E_n\| &\le & \sum\limits_{j=0}^{N}\|\tilde{T}_j\|\|d_{n-j}\|+\sum\limits_{j=N+1}^{n}\|\tilde{T}_j\|\|d_{n-j}\|\\
&<& \frac{\epsilon}{2(K+1)}\sum\limits_{j=0}^{N}\|\tilde{T}_j\|+M\sum\limits_{j=N+1}^{n}\|\tilde{T}_j\|\\
&<& \frac{\epsilon}{2(K+1)}K+M\frac{\epsilon}{2M} \leq \epsilon.
\end{eqnarray*}
This proves $\lim\limits_{n\to\infty}C_n=\lim\limits_{n\to\infty}X_n=T(a_1,\ldots,a_k,b).$
\end{proof}

\noindent{\bf\it Proof of Proposition \ref{prop71}}:
Without loss of generality, we can assume that the interval $I=(-1,1)$ and $t_0=0.$
Since $t\mapsto A(t)$ is real analytic at $t=0,$ there exists an $r>0$ such that
$A(t)$ can be expressed as $A(t)=A(0)+\sum\limits_{j=1}^{\infty}C_jt^j$
for all $|t|<r.$

Here $\sum\limits_{j=1}^{\infty}C_jt^j$ is absolutely convergent for $|t|<r.$
Let $f(A)=A^{1/2}$ be the square root map.
Since each $k\textrm{th}$ order derivative $D^kf(A(0))$ is $k$-linear and bounded,
by $\text{Lemma } \ref{lema1}$
we have
\begin{align*}
D^kf(A(0))(A(t)-A(0),\ldots, & A(t)-A(0)) \\
&=\sum\limits_{n=0}^{\infty}\sum\limits_{j_1+\cdots+j_k=n}t^n D^kf(A(0))(C_{j_1},\ldots,C_{j_k})
\end{align*}

Let $B_{k,n}$ denote the matrix $\sum\limits_{j_1+\cdots+j_k=n}D^kf(A(0))(C_{j_1},\cdots,C_{j_k}).$
For $n<k,$ $B_{k,n}$ be the zero matrix.
We have the following Taylor expansion  of $f$ at $A(0)$ in a neighbourhood $U \subseteq \mathbb{P}(m)$. See \cite{mn}.
$$f(A)=f(A(0))+\sum\limits_{k=1}^{\infty}\frac{1}{k!}D^kf(A(0))(A-A(0),\cdots,A-A(0)).$$
Let $\lambda_0$ be the minimum eigenvalue of $A(0).$
Since $\lambda_0>0,$ the square root function $f$ is real analytic at $\lambda_0,$ i.e.,
there exists an $r_0>0$ such that the series
$\sum\limits_{k=1}^{\infty}\frac{1}{k!}f^{(k)}(\lambda_0)(t-\lambda_0)^k$ is absolutely and locally uniformly convergent in  $(\lambda_0-r_0,\lambda_0+r_0).$
Choose $\delta,$ $0<\delta<r$ such that $\sum\limits_{j=1}^{\infty}\|C_j\|\delta^j < r_0$ and $A(t)\in U$ for all $t\in (-\delta,\delta).$
Thus for all $|t|<\delta,$
\begin{equation}
f(A(t))=f(A(0))+\sum\limits_{k=1}^{\infty}\frac{1}{k!}\sum\limits_{n=k}^{\infty}B_{k,n}t^n.\label{eqa1}
\end{equation}
We show that the iterated sum $\sum\limits_{k=1}^{\infty}\frac{1}{k!}\sum\limits_{n=k}^{\infty}\|B_{k,n}\||t|^n<\infty.$
Let $C$ be the sum $\sum\limits_{j=1}^{\infty}\|C_j\|\delta^j.$
For $|t|<\delta,$ we have
\begin{align*}
& \sum\limits_{n=k}^{\infty}\|B_{k,n}\||t|^n\le\sum\limits_{n=k}^{\infty}\|B_{k,n}\|\delta^n\\
& \le \sum\limits_{n=k}^{\infty}\sum\limits_{j_1+\cdots+j_k=n}\|D^kf(A(0))(C_{j_1},\ldots,C_{j_k})\|\delta^n\\
&\le \|D^kf(A(0))\|\sum\limits_{n=k}^{\infty}\sum\limits_{j_1+\cdots+j_k=n}(\|C_{j_1}\|\delta^{j_1})\cdots (\|C_{j_k}\|\delta^{j_k})\\
&= \|D^kf(A(0))\|C^k.
\end{align*}
The last equality follows from the convergence of Cauchy product of the series $\sum\limits_{j=1}^{\infty}\|C_j\|\delta^j.$
By \cite{bss}
$$\|D^kf(A(0))\|=\|f^{(k)}(A(0))\|=|f^{(k)}(\lambda_0)|.$$
For $|t|<\delta,$ we have $C<r_0$ and hence
$$\sum\limits_{k=1}^{\infty}\frac{1}{k!}\sum\limits_{n=k}^{\infty}\|B_{k,n}\||t|^n\le\sum\limits_{k=1}^{\infty}\frac{1}{k!}|f^{(k)}(\lambda_0)|C^k<\infty.$$
This implies that the iterated sum on the right hand side of \eqref{eqa1} is equal to the sum $\sum\limits_{n=1}^{\infty}\sum\limits_{k=1}^{n}\frac{1}{k!}B_{k,n}t^n.$
This shows that $\sqrt{A(t)}$ can be expressed as the power series
$$\sqrt{A(t)}=\sqrt{A(0)}+\sum\limits_{n=1}^{\infty}\left(\sum\limits_{k=1}^{n}\frac{1}{k!}B_{k,n}\right)t^n\textrm{ for all }|t|<\delta.$$
\qed

\vskip.2in
{\bf\it{Acknowledgement}}: The first author is supported by the SERB MATRICS grant MTR/2018/000554.

\end{document}